\documentclass[11pt,leqno]{amsart}
\usepackage[breaklinks=true]{hyperref}
\usepackage{amsthm,amsfonts,amssymb,amsmath,oldgerm}
\numberwithin{equation}{section}
\usepackage{fullpage}
\usepackage{color}
\usepackage{calrsfs}
\DeclareMathAlphabet{\pazocal}{OMS}{zplm}{m}{n}
\usepackage{graphicx, xcolor}
\usepackage{subcaption}

\graphicspath{{draws/}}

\def\eps{\varepsilon }
\def\e{\varepsilon}

\newcommand\R{\mathbb R}

\def\eps{\varepsilon}
\def\e{\varepsilon}

\newcommand\br{\begin{remark}}
\newcommand\er{\end{remark}}
\newcommand\bp{\begin{pmatrix}}
\newcommand\ep{\end{pmatrix}}
\newcommand{\be}{\begin{equation}}
\newcommand{\ee}{\end{equation}}
\newcommand\ba{\begin{equation}\begin{aligned}}
\newcommand\ea{\end{aligned}\end{equation}}

\newcommand{\bap}{\begin{app}}
\newcommand{\eap}{\end{app}}
\newcommand{\begs}{\begin{exams}}
\newcommand{\eegs}{\end{exams}}
\newcommand{\beg}{\begin{example}}
\newcommand{\eeg}{\end{exaplem}}
\newcommand{\bpr}{\begin{proposition}}
\newcommand{\epr}{\end{proposition}}
\newcommand{\bt}{\begin{theorem}}
\newcommand{\et}{\end{theorem}}
\newcommand{\bc}{\begin{corollary}}
\newcommand{\ec}{\end{corollary}}
\newcommand{\bl}{\begin{lemma}}
\newcommand{\el}{\end{lemma}}
\newcommand{\bd}{\begin{definition}}
\newcommand{\ed}{\end{definition}}
\newcommand{\brs}{\begin{remarks}}
\newcommand{\ers}{\end{remarks}}

\newcommand{\const}{\text{\rm constant}}

\newcommand{\Span}{{\rm Span }}

\newcommand{\sgn}{\text{\rm sgn}}
\newtheorem{theorem}{Theorem}[section]
\newtheorem{proposition}[theorem]{Proposition}
\newtheorem{corollary}[theorem]{Corollary}
\newtheorem{lemma}[theorem]{Lemma}

\theoremstyle{remark}
\newtheorem{remark}[theorem]{Remark}
\theoremstyle{definition}
\newtheorem{definition}[theorem]{Definition}

\newtheorem{example}[theorem]{Example}

\newcommand{\beq}{\begin{equation}}
\newcommand{\eeq}{\end{equation}}

\hypersetup{
    colorlinks=True,
    linkcolor=.,
    filecolor=.,    
    citecolor=.,
    urlcolor=blue,
    pdfpagemode=FullScreen,
    }

\title{Nonconvex optimization and convergence of stochastic gradient descent, 
and solution of asynchronous games}

\author{Jessica Babyak}
\address{Indiana University, Bloomington, IN 47405}
\email{jt103@iu.edu}
\thanks{Research of J.B. was partially supported under NSF grants no. DMS-2154387 and DMS-2206105}
\author{Kevin Buck}
\address{Indiana University, Bloomington, IN 47405}
\email{kevbuck@iu.edu}
\thanks{Research of K.B. was partially supported under NSF grants no. DMS-2154387 and DMS-2206105}
\author{Paolo Piersanti}
\address{The Chinese University of Hong Kong, Shenzhen}
\email{ppiersanti@cuhk.edu.cn}
\thanks{Research of P.P. was partially supported by a Zorn Postdoc at Indiana University, Bloomington.} 
\author{Kevin Zumbrun}
\address{Indiana University, Bloomington, IN 47405}
\email{kzumbrun@iu.edu}
\thanks{Research of K.Z. was partially supported under NSF grants no. DMS-2154387 and DMS-2206105}
\author{Christiane Gallos}
\address{Indiana University, Bloomington, IN 47405}
\email{chgallos@iu.edu}
\thanks{Research of C.G. was partially supported under NSF grants no. DMS-2154387 and DMS-2206105}
\author{Dorothea Gallos}
\address{Indiana University, Bloomington, IN 47405}
\email{dgallos@iu.edu}
\thanks{Research of D.G. was partially supported under NSF grants no. DMS-2154387 and DMS-2206105}

\begin{document}

\begin{abstract}
We review convergence and behavior of stochastic gradient descent for convex and nonconvex optimization, 
establishing various conditions for convergence to zero of the variance of the gradient of the 
objective function, and presenting a number of simple examples demonstrating the approximate 
evolution of the probability density under iteration, including applications to both classical
two-player and asynchronous multiplayer games.
\end{abstract}

\date{\today}
\maketitle
\tableofcontents

\section{Introduction}\label{s:intro}
The purpose of this note is to review in a mimimalist setting the method of stochastic gradient descent,
and its application to nonconvex optimization, giving in the process an elementary proof of convergence
in probability of the resulting stochastic approximants to the set of critical points of the objective function
under mild standard assumptions.
Our particular interest is in determining conditions on step size needed for convergence in convex vs. nonconvex case.
At the same time, we present a number of illustrative example, including
a (to our knowledge) novel application to solution of multiplayer games.

\subsection{Gradient descent (GD)}\label{s:GD}
Consider an objective function to be minimized
\be\label{object}
\hbox{\rm $f\in C^2:\R^d \to \R$,
without loss of generality $f\geq 0$.}
\ee
The method of gradient descent (GD) consists of the iteration 
\be\label{GD}\tag{GD}
w_{m+1}= w_m -\alpha_m \nabla f(w_m),
\ee
where $\alpha_m\geq 0$ are step sizes to be chosen depending on the particular implementation of \eqref{GD}.
Here, we will assume always that the sequence $\{\alpha_m \}$ is predetermined, and monotone nonincreasing.

For step size fixed and sufficiently small, we have the following standard convergence result.

\begin{proposition}\label{detprop}
	\label{constGD}
	Assuming the Hessian bound $|\nabla^2 f|\leq L$,	
	and taking $\alpha_j=\alpha \equiv \const$ with $\alpha <2/L$, we have for any solution $\{w_m\}$
	of \eqref{GD} that (i) $f(w_m)$ is monotone decreasing, and (ii) $\nabla f(w_m)\to 0$ as $m\to\infty$.
	If, also, $|f(w)|\to \infty$ as $|w|\to \infty$, then
	(iii) $w_m$ converges as $m\to\infty$ to the set $\mathcal{C}:=\{w:\, \nabla f(w)=0\}$ 
	of critical points of $f$.
\end{proposition}

\begin{proof} 
	By Taylor's theorem, with remainder, we have for some $\tilde w$ on line segment $\overline{w_m,w_{m+1}}$:
	\ba\label{base}
		f(w_{m+1})-f(w_m)  &= \nabla f(w_m)(w_{m+1}-w_m) + 
	(w_{m+1}-w_m)^T \nabla^2 f(\tilde w) (w_{m+1}-w_m) \\
		&\leq - \alpha |\nabla f(w_m)|^2+ (L\alpha^2/2) |\nabla f(w_m)|^2\\
		&\leq - \alpha |\nabla f(w_m)|^2(1- L\alpha/2),
		\ea
	which for $\alpha<2/L$ is strictly negative. 
	This establishes (i). Moreover, summing the left and right sides of \eqref{base}
	and noting that the lefthand quantity is a telescoping sum, we find that
	$ f(w_{M})-f(w_1)  = -\alpha (1-L\alpha/2) \sum_{m=1}^{M-1} |\nabla f(w_m)|^2$,
	hence, by positivity of $f$,
	\be\label{basesum}
	\sum_{m=1}^{M-1} |\nabla f(w_m)|^2 \leq \frac{f(w_1)}{ \alpha (1-L\alpha/2) }<\infty.
	\ee
	From convergence of \eqref{basesum}, we then obtain $|\nabla f(w_m)|\to 0$, or (ii).
	Finally, observing that $|f|\to \infty$ as $|w|\to \infty$, together with (i), gives uniform
	boundedness $|w_m| \leq M$ for some $M>0$, 
	we obtain (iii) from (ii) together with continuity of $\nabla f$/compactness
	of $\{w:\, |w|\leq M\}$.
\end{proof} 

\subsection{Stochastic gradient descent (SGD)}\label{s:SGD}
Next, consider the same type of objective function \eqref{object}, augmented with a
{\it stochastic gradient estimator} $\tilde \nabla f(w)$, satisfying
\be\label{enabeq}
E[\tilde \nabla f(w)|w]=\nabla f(w).
\ee
The method of stochastic gradient descent (SGD) consists of the stochastic iteration
\be\label{SGD}\tag{SGD}
w_{m+1}= -\alpha_m \tilde \nabla f(w_m)
\ee
obtained by replacing in \eqref{GD} the exact gradient $\nabla f(w)$ 
by the randomly estimated $\tilde \nabla f(w)$.
Here, the idea is that the gradient estimator should be cheaper to compute than the actual gradient.

\medskip \noindent
\textbf{Canonical example.} A concrete example, from which the method originates, 
is minimization of an objective function in the form of a sum\footnote{Without loss of generality written
as an average.}
\be\label{canform}
f(w)=(1/N)\sum_{i=1}^N f_i(w), 
\ee
with $N$ large. Fixing a ``batch size'' $b\ll N$, we may define the batch sample
\be\label{sample}
\tilde f(w):=(1/b) \sum_{i\in S} f_i(w),
\ee
where the subset $S\subset \{1,\dots, N\}$ is chosen with equal likelihood among samples of size $|S|=b$.
A natural gradient estimator, satisfying \eqref{enabeq} by definition, is then 
\be\label{canest}
\tilde \nabla f(w):= \nabla \tilde f(w).
\ee
Evidently, $\tilde \nabla f$ is considerably cheaper to compute than $\nabla f$; In practice, $m$ may well be $1$.

Such problems arise in statistical estimation and machine learning.
One may think of the functions $f_i(w)$
as measuring ``goodness'' of fit at a data point $i$, under the choice of parameters $w\in \R^d$.
For example, a particularly familiar example is given by the least squares error
\be\label{LS}
f_i(w)= (1/2)|y_i- \phi(x_i, w)|^2,
\ee
where $\phi(\cdot,w)$ is a function fitting data set $(x_i,y_i)$, $i=1,\dots,N$.

In machine learning applications, $w$ corresponds to a choice of weights in a neural net, hence
the choice of variable name $w$.  But, in general, the parameters $w$ could have a variety of interpretations.
Likewise, the function $f$ in \eqref{SGD} need not be of form \eqref{canform}, but only possess
an inexpensive gradient estimator $\nabla f$.
And, this gradient estimator need not correspond as in \eqref{canest} to the gradient of some primitive estimator,
but only satisfy the consistency condition \eqref{enabeq}.

\medskip \noindent
\textbf{Stochastic coordinate descent.} Another standard example, applying to objective functions
of general form $f(w)$, $w\in \R^d$ is stochastic coordinate descent (SCD), in which gradient descent 
is performed randomly in one coordinate direction $w_j$ at a time, that is, taking
\be\label{scd}
\tilde \nabla f(w):=d \sum_{j=1}^d \theta_j (\partial f/\partial w_j)(w),
\ee
where $\theta=(\theta_1, \dots, \theta_d)$ is a random variable equal to one of the standard coordinate directions 
$e_j$ with equal probability $1/d$.
This is particularly helpful if $f$ has a decoupled summation structure $f(w)=\sum_j f_j(w_j)$
or otherwise breaks into blocks with sparse dependence on coordinates of $w$.

\medskip \noindent
\textbf{Assumptions and verification.} 
We shall make in various combinations the assumptions
\be\label{small}
\alpha_j \ll 1,
\ee
\be\label{suminf}
\sum \alpha_m=\infty,
\ee
and
\be\label{sumfin}
\sum \alpha_m^2<\infty.
\ee
on the step size $\alpha_m\geq 0$.

We assume always a uniform Hessian bound 
\be\label{Hess}
|\nabla^2 f(w_m)|\leq L
\ee
on the function $f$, and a uniform variance bound
\be\label{var}
E[ |\tilde \nabla  f(w_m)|^2 -|\nabla f(w_m)|^2]\leq \sigma^2
\ee
on the gradient estimator $\tilde \nabla f$. 

The following straightforward results verify \eqref{Hess}-\eqref{var} for our canonical examples. 

\begin{proposition}\label{aver}
	For $f$ as in \eqref{canform}, $f_i$ uniformly bounded in $C^2$ norm on compact sets and
	$\tilde \nabla f:=\nabla \tilde f$, assumptions \eqref{Hess}-\eqref{var} 
	hold on $\{w:\, |w|\leq K\}$ for any $K>0$.
\end{proposition}

\begin{corollary}\label{acor}
	For $f$ as in \eqref{canform} and $f_i$ as in \eqref{LS}, with $\phi \in C^2$,
	and $\tilde \nabla f:=\nabla \tilde f$, assumptions \eqref{Hess}-\eqref{var} 
	hold on $\{w:\, |w|\leq K\}$ for any $K>0$.
\end{corollary}

\begin{proposition}\label{averscd}
	For $f$ uniformly bounded in $C^2$ norm on compact sets and $\tilde \nabla f$ as in \eqref{scd},
	assumptions \eqref{Hess}-\eqref{var} hold on $\{w:\, |w|\leq K\}$ for any $K>0$.
\end{proposition}

\br
Proposition \ref{aver} (Corollary \ref{acor}), verify \eqref{Hess}-\eqref{var} on bounded sets.
To apply our analysis below, based on these assumptions, it is thus sufficient to show that
$w_m$ remains bounded almost surely.  For example, this follows for \eqref{canform} under the auxiliary assumption
\be\label{auxcond}
\hbox{\rm
$w \cdot \tilde \nabla f_i(w)\geq \theta |w|^2$ for all $i$, some $\theta>0$ and $|w|$ sufficiently large,}
\ee
and for \eqref{scd} under 
\be\label{auxcondscd}
\hbox{\rm
$w_i (\partial f/\partial_{w_i}) \geq \theta |w_i|^2$ for all $i$, some $\theta>0$ and $|w_i|$ sufficiently large.}
\ee
\er

\subsection{Main results}
Our main conclusions are the following two generalizations of Proposition \ref{detprop} to the stochastic case.
The first concerns general, possibly nonconvex and the second convex or ``convex-like'' $f$.

\begin{theorem}\label{main}
	For nonnegative $f\in C^2$ and $\alpha_m$ satisfying \eqref{small}-\eqref{var}, 
	we have for any random variables $\{w_m\}$ satisfying \eqref{SGD} that
	(i) $E[f(w_m)]\to \liminf_{m\to \infty} E[f(w_m)]$ as $m\to \infty$, and
	(ii) $E[|\nabla f(w_m)|^2]\to 0$ as $m\to\infty$.
	If, also, $|f(w)|\to \infty$ as $|w|\to \infty$, then
	(iii) $w_m$ converges in probability to the set $\mathcal{C}:=\{w:\, \nabla f(w)=0\}$ 
	of critical points of $f$.
\end{theorem}

We can say much more for functions $f$ satisfying the ``approximate convexity'' condition
\be \label{acon}
|f(w)|\sim |\nabla f(w)|^2 \sim |w|^2.
\ee
This includes uniformly convex functions with bounded Hessian and minimum value zero.

\begin{proposition}\label{main2}
	For nonnegative $f\in C^2$ and nonincreasing $\alpha_m$ satisfying \eqref{small}, \eqref{suminf}, 
	\eqref{Hess}, \eqref{var}, and approximate convexity, \eqref{acon}, 
	plus $\lim_{m\to \infty}\alpha_m=0$,
	any random variables $\{w_m\}$ satisfying \eqref{SGD} with $E[f(w_1)]$ finite converge in probability
	to $0$ as $m\to \infty$, with also 
	\be\label{concon}
	\hbox{\rm $E[|w_m|^2]$, $E[f(w_m)]$ and $E[\nabla f(w_m)|^2]\to 0$.}
	\ee
\end{proposition}

The conditions \eqref{suminf} and $\lim_{m\to \infty}\alpha_m=0$ are sharp, as  
seen by the explicit (convex) example of Section \ref{s:simple}, Remark \ref{sharprmk}.
Indeed, \eqref{suminf} (but not $\lim_{m\to \infty}\alpha_m=0$)
is necessary even in the deterministic case as shown in Proposition \ref{necprop} below,
while $\alpha_m$ quantifies the size of stochastic effects, hence gives a lower bound on accuracy
of approximations if $\lim_{m\to \infty}\alpha_m\neq 0$.
It is not clear whether \eqref{sumfin} is sharp in the nonconvex case or a technical assumption.
Based on our experiments with simple models, we expect it is a technical assumption.
A partial result in the absence of \eqref{sumfin} is given in Proposition \ref{rvprop} below.
In the special case of stochastic coordinate descent, we show in Proposition \ref{suffpropscd}
convergence assuming only $\sum_j \alpha_j=\infty$ and $\alpha_m$ sufficiently small, similary 
as in the deterministic case.

The simple proofs of Theorem \ref{main} and Proposition \ref{main2} at least
appear to be new, and perhaps the results as well- we make no claim to novelty in the latter regard.
We give in Sections \ref{s:egs} and \ref{s:FP} some simple examples illustrating and further illuminating
these conclusions, based in part on explicit solutions and in part on numerical Monte Carl and Fokker-Planck
approximations.

\medskip
{\bf Application to games.} In Section \ref{s:games}, we describe an (SGD) approach to 
2-player games, and $(n-1)$ vs. $1$-player games with asynchronous coalition as defined in \cite{BBDJZ},
based on $\ell^p$ smoothing of the maximum function, and carry out numerical experiments for some simple examples.

\subsection{Time-averaging and adaptive step size}
Finally, we mention two interesting and frequently used variants improving performance.
The first, sidestepping the technical issues above while further stabilizing (SGD), 
consists in ``filtering'', or time-averaging the output of the basic algorithm \eqref{SGD}.
Namely, saving the outputs at intermediate steps, define at step $m$ an averaged variable $z_m$
taking values $w_j$, $j=1,\dots, m$ with probabilities
\be\label{aprob}
P(z_m=w_j)= \frac{\alpha_j}{\sum_{j=1}^m \alpha_j}, \qquad j=1,\dots, m.
\ee
Then, we have the following standard result,\footnote{See for example, lecture notes \cite{Co}.}
not requiring \eqref{sumfin}, valid for general (nonconvex) $f$.

\begin{proposition}[Co,SGD]\label{avprop}
	For nonnegative (possibly nonconvex) $f\in C^2$ and nonincreasing $\alpha_m$ 
	satisfying \eqref{small}, \eqref{suminf}, 
	\eqref{Hess}, \eqref{var}, and $\lim_{m\to \infty}\alpha_m=0$,
	and random variables $\{w_m\}$ satisfying \eqref{SGD}, $z_m$ defined
	in \eqref{aprob} satisfies 
	\be\label{averaged}
	\hbox{\rm $ E[|\nabla f(z_m)|^2]\to 0$ as $m\to\infty$.}
	\ee
\end{proposition}

See Remark \ref{comprmk} comparing the argument of Proposition \ref{avprop} to that of Theorem \ref{main}.
The second variant is the use of adaptive time steps \cite{A}, as is important even in the deterministic case for
faster convergence. We do not treat this, as out of our present scope. See, for example, \cite{LO}.

\subsection{Discussion and open problems}
The investigations recorded in this note were motivated by a desire to apply stochastic gradient
descent techniques originating in machine learning/statistical estimation \cite{RM} 
to the study of large nonconvex optimization problem arising in ``asynchronous multiplayer games'' \cite{BBDJZ}
or other such general applications.
As such, our approach is from a naive perspective, abstracting general aspects 
that might be applied in a wider setting.

Our main theoretical results show that square summability of step size $\alpha_m$ is sufficient for convergence to
the set of critical points without ``filtering'', or time-averaging, and in a number of cases 
just $\alpha_m \to 0$ as $m\to \infty$. It is a very interesting question whether the latter might suffice
in all cases, both mathematically and practically: more generally, whether mere sufficient smallness
of $\alpha_m$ might imply convergence to an arbitrarily small neighborhood of the critical set.
For, as demonstrated in our experiments of Section 
\ref{s:gexp}, a small constant step size appears more convenient and effective in practice.
Likewise it is interesting to know when filtering may be dispensed with.

Our main application is an adaptation to two- and multiplayer games or more generally any minimax problem
$ \min_{y\in Y\subset \R^d} \max_{0\leq j\leq N}\{\phi_j(y)\}$, $\phi_j\geq \theta>0$,
noting that this may be smoothly approximated as $\min_{y\in Y\subset \R^d} \|\phi\|_{\ell^p}$,
$p\gg 1$.
Minimizing instead $\|\phi\|_{\ell^p}^p= \sum_j |\phi_j|^p$,
we convert to a problem of form \eqref{canform} to which (SGD) and or (SCD) may be applied.
Our experiments in Section \ref{s:gexp} show good performance for small example problems and 
reasonably sized smoothing exponent $p$.
As we discuss there, the success of this method for large problems would appear to 
require further elaboration such as multi-grid iteration/rescaling of payoff functions as $p$ is increased.
Nonetheless, it seems an intriguing variation in the direction of interior point methods with iterated smoothing.

We note that for problems \eqref{canform} arising in machine learning, the index $i$ in $f(w)=\sum_i f_i(w)$
represents instances of a training set, and the coordinates $w_j$ of $w$ weights in a neural net.
For deep learning applications, the number of weights typically ranges from one tenth to one times the size 
of the training set, with one-tenth considered somewhat optimal under the ``rule of ten''.
For a classical two-player game, the dimension $N$ of $w$ is equal to the number of elements $\phi_j$.
For a three-player asynchronous game on the other hand, as described in Section \ref{s:games},
the dimension of $w$ is typically $N^2$, whereas the number of $f_i$ in $\sum_i f_i$ is $N$, with $N\gg 1$.
Thus, it is in a rather different regime, what would be ``overfitting'' (``or undertraining'') 
in the context of machine learning.
The implications of this discrepancy for performance in our context is an interesting open question.

Finally, we note that our convergence results do not distinguish between global and local minima, or 
saddlepoints, and indeed our investigations in Section \ref{s:nonconvex} show for simple examples
that (SGD) may be trapped with nonzero probability at local minima.
Incorporation of annealing or multigrid methods, though out of our current scope, may be expected to
be extremely important for treatment of large games or other applications, as is the neglected topic
of adaptive step size.
The treatment of large multiplayer games in particular seems a very interesting problem for further
investigation.


\medskip\noindent
\textbf{Acknowledgement} These notes are the product of a reading group carried out in parallel with an 
REU project on optimal asynchronous coalitions in multi-player games,\footnote{Examples of nonconvex 
optimization problems \cite{BBDJZ}.} supported by NSF grant number DMS-2051032 (REU).
We thank the National Science Foundation and Indiana University for their infrastructural support.
Thanks also to L. Miguel Rodrigues for a helpful conversation regarding physical connections/background
to do with Fokker-Planck approximation.
All code for this paper is done in python using standard packages, and can be found at \href{https://github.com/kevmbuck/SGD}{https://github.com/kevmbuck/SGD}.

\section{Variable step-size deterministic case}\label{s:det}
We start by studying convergence of the deterministic gradient descent algorithm \eqref{GD}
with variable step size, under assumptions \eqref{suminf} and \eqref{Hess}, together with \eqref{small}.

\subsection{Continuous-time analog}\label{s:cont}
It is instructive to consider the analogous continuous gradient descent flow
\be\label{contGD}
\dot w(t)= -\alpha(t)\nabla f(w(t)),
\ee
with varying rate $\alpha(t)$.
From the computation $ \dot f(w)= -\alpha(t)|\nabla f(w(t))|^2, $
we obtain, integrating in time, the standard energy estimate
$$
\int_0^T \alpha(t) |\nabla f(w(t))|^2 dt= F(0)-F(T)\leq F(0),
$$
giving an averaged decay result for $|\nabla f(w))|$ so long as 
\be\label{continf}
\int_0^{+\infty} \alpha(t) dt=\infty.
\ee

Indeed, by the change of time-coordinate $dt/d\tau= 1/\alpha(t)$, we may convert \eqref{contGD} to the 
constant-rate case
$$
d w/d\tau = -\nabla f(w(\tau)), \qquad  \tau\in [0, \int_0^\infty (1/\alpha(t)) dt].
$$
From this, we see immediately that \eqref{continf} is necessary for convergence to equlibrium, for
which $\tau$ must go to infinity.
Likewise, the resulting energy estimate $ \int_0^T  |\nabla f(w(\tau))|^2 dt\leq F(0)$
gives $|\nabla f(w)|\to 0$ as $t$ (hence also $\tau$) goes to infinity, under standard mild conditions
giving also control of $|(d/dt)\nabla f(w)|$.
These observations give a useful guide to the discrete case as well. In particular, the idea of time
rescaling may be seen to underly our ultimate proof of (discrete) convergence.

\subsection{Necessity}\label{s:nec}
We first address necessity of our conditions, for which we obtain readily the following definitive result,
applying to general $f$, not necessarily convex.

\begin{proposition}\label{necprop}
For $f\in C^2$ and $\alpha_m$ satisfying \eqref{Hess} and \eqref{small}, condition \eqref{suminf} is necessary in
order that $|\nabla f(w_m)|\to 0$ as $m\to \infty$ for all solutions of \eqref{GD} such that $\nabla f(w_1)\neq 0$.
\end{proposition}

\begin{proof}
From the first-order Taylor expansion
$
\nabla f(w_{m+1})= \nabla f(w_m) -\alpha_m \nabla^2 f(\tilde w) \nabla f(w_m)
$
together with \eqref{Hess} we have, applying the reverse triangle inequality and using that $\alpha_m,L\geq 0$,
$$
|\nabla f(w_{m+1})|\geq  |\nabla f(w_m)| -\alpha_m L |\nabla f(w_m)|= (1-L\alpha_m) |\nabla f(w_m)|.
$$

By induction, we obtain therefore
$
|\nabla f(w_M)|\geq \Pi_{m=1}^{M-1} (1- L \alpha_m) |\nabla f(w_1)|,
$
	or, taking logarithms, and using the first order Taylor expansion of $\log$, together with
	smallness assumption \eqref{small}:
\ba\label{logeq}
	\log( |\nabla f(w_M)|)&\geq \sum_{m=1}^{M-1} \log (1- L \alpha_m) + \log( |\nabla f(w_1)|)\\
	&\geq  -\sum_{m=1}^{M-1} 2 L \alpha_m + \log( |\nabla f(w_1)|).
\ea

	Now, suppose that $\nabla f(w_1)\neq 0$, so that $\log( |\nabla f(w_1)|$ is finite,
	and $|\nabla f(w_M)|\to 0$ as $M\to \infty$, so that $\log(|\nabla f(w_M)|)\to - \infty$.
	Then, equating left and right hand limits in \eqref{logeq}, we must have 
	$-\sum_{m=1}^{M-1} 2 L \alpha_m=-\infty$, or $\sum_{m=1}^{M-1}  \alpha_m=\infty$.
\end{proof}

\subsection{Sufficiency}\label{s:suff}
We now show that our conditions are sufficent for convergence to zero of $\nabla f(w_m)$, hence
convergence of $w_m$ to the critical set $\mathcal{C}$ of critical points of $f$, for $f$ not
necessarily convex.

\begin{proposition}\label{suffprop}
	For $f\in C^2$ and $\alpha_m$ satisfying \eqref{small}, \eqref{suminf}, and \eqref{Hess}, 
	we have for any solution of \eqref{GD} that (i) $f(w_m)$ is monotone decreasing, and (ii) $\nabla f(w_m)\to 0$ as $m\to\infty$.
	If, also, $|f|\to \infty$ as $|w|\to \infty$, then 
	(iii) $w_m$ converges as $m\to\infty$ to the critical set $\mathcal{C}:=\{w:\, \nabla f(w)=0\}$. 
\end{proposition}

\begin{proof}
First, observe that assuming \eqref{small} and \eqref{Hess}, we have by Taylor expansion
\be\label{Taylor_est}
f(w_{m+1})\leq f(w) - (1/2)\alpha_m|\nabla f|^2,
\ee
whence, by telescoping sum,
$f(w_{M+1})\leq  f(w_1)  - \sum_{m=1}^M \alpha_m |\nabla f(w_m)|^2,$
giving
\be\label{sumcon}
\sum_{m=1}^M \alpha_m |\nabla f(w_m)|^2< \infty.
\ee

Provided $\alpha_j$ are sufficiently small, and $\sum \alpha_j=\infty$, we may
choose a sequence $m_j\to \infty$ such that 
\be\label{o1}
1/2< \sum_{m_j+1}^{m_{j+1}} \alpha_m \leq 1
\ee
for all $j$.
Thus, 
$$
\sum_{m_j+1}^{m_{j+1}} \alpha_j |\nabla f(w_m)|^2 \geq (1/2) \inf_{m_j<m\leq m_{j+1}}|\nabla f(w_m)|^2,
$$
whereas by \eqref{sumcon}, the lefthand side goes to zero as $j\to \infty$.  It follows that
\be\label{infcon}
\inf_{m_j<m\leq m_{j+1}}|\nabla f(w_m)|^2 \to 0
\ee
as $j\to \infty$.

On the other hand, for $m, n\in [m_j+1,m_{j+1}]$, $m<n$,
$
|w_n-w_m| \leq \sum_{j=m+1}^{n} \alpha_j |\nabla f(w_j)|,
$
hence, by \eqref{Hess},
	\be\label{tempo}
|\nabla f(w_n)-\nabla f(w_m)|^2\leq L^2
\Big(\sum_{j=m+1}^{n} \alpha_j |\nabla f(w_j)|\Big)^2.
\ee
By Jenssen's inequality, noting that, by \eqref{o1},
$\sum_{j=m+1}^{n} \alpha_j |\nabla f(w_j)|$ is approximately a weighted average of $|\nabla f(w_j)|$,
	the righthand side of \eqref{tempo}
is less than or equal to a bounded multiple of
$$
\sum_{j=m+1}^{n} \alpha_j |\nabla f(w_j)|^2.
$$
Noting that the latter goes to zero as $m,n\to \infty$, by \eqref{sumcon}, we thus have
$$
\max_{m, n\in [m_j+1,m_{j+1}]} |\nabla f(w_n)-\nabla f(w_m)|^2 \to 0
$$
as $j\to \infty$, which, together with \eqref{infcon}, gives $|\nabla f(w_n)|\to 0$ as $n\to\infty$ as claimed.
	The remaining assertions then follow exactly as in the proof of Proposition \ref{detprop}.
\end{proof}

\br\label{detrmk}
For fixed step size $\alpha_m\equiv \const$, we obtain $|\nabla f(w_m)|^2\to 0$ immediately from \eqref{sumcon}.
However, even in this deterministic case, the argument for decreasing step size is somewhat subtle,
relying on the intuition afforded by Section \ref{s:cont} and the analogy to continuous flow.
\er

\begin{proof}[Alternative proof of Proposition \ref{suffprop}] 
An alternate proof is to note that away from the set 
$$
\mathcal{C}^\eps:=\{ w:\, |\nabla f(w)|\leq \eps\},
$$
	$|\nabla f|\geq \eps$ and so, observing by \eqref{Taylor_est} combined with \eqref{GD} that 
$$
	|f(w_{m+1})-f(w_m)|\geq (\alpha_m/2)|\nabla f(w_m)|^2= (1/2) |\nabla f(w_m)| |w_{m+1}-w_m|,
$$
together with monotone decrease of $f$, yielding by $f\geq 0$ finite oscillation in $f$,
we find that the number of times that $w_m$ leaves $\mathcal{C}^\eps$ by distance of $\eta>0$ and then returns must
be finite.  But, at the same time, the number of times that $w_m$ visits $\mathcal{C}^\eps$ must be infinite, or else 
$$
\sum_{m=1}^\infty \alpha_j |\nabla f(w_j)|^2\geq \eps^2 \sum_{j=J}^\infty \alpha_j=\infty,
$$
contradicting \eqref{sumcon}.
Thus, eventually $w_m$ stays within $\eta$ of $\mathcal{C}^\eps$. Since $\eps>0$ was arbitrary, this
proves that $|\nabla f(w_m)|\to 0$ as $m\to \infty$, or assertion (i).  Assertions (ii)-(iii) then follow
as previously.
\end{proof}

\section{The stochastic case}\label{s:sgd}
With slight modification, and the additional hypotheses \eqref{sumfin} and \eqref{var}, 
the above deterministic argument gives convergence also in the stochastic case, as we now show.

\subsection{Key estimates}\label{s:key}
We start with the following key estimates, following \cite{GG}.

\begin{lemma}\label{keylem}
	For nonnegative $f\in C^2$ and $\alpha_m$ satisfying \eqref{small}, \eqref{Hess}, and \eqref{var},
	solutions of \eqref{SGD} with finite initial expectation $E[f(w_1)]$ satisfy
	\ba\label{keyeq}
	- (2\alpha_m)E[|\nabla f(w_m)|^2] - \sigma^2 (L/2) \alpha_m^2 &\leq
	E[f(w_{m+1})]- E[f(w_{m})]\\
	&\leq
	- (\alpha_m/2)E[|\nabla f(w_m)|^2] + \sigma^2 (L/2) \alpha_m^2
	\ea
	and
	\be\label{keysum}
	\sum_{m=1}^{M-1} \alpha_m E[|\nabla f(w_m)|^2] \leq 2E[f(w_{1})] +\sigma^2 L \sum_{m=1}^{M-1} \alpha_m^2.
	\ee
\end{lemma}

\begin{proof}
	By Taylor's theorem, with remainder, we have, similarly as in \eqref{base}, 
	\ba\label{sbase}
	f(x_{m+1})- f(x_{m}) &\leq \nabla f(x_m) \cdot (x_{m+1}- x_{m}) + (L/2) |x_{m+1}- x_{m}|^2 \\
	&= 
	-\alpha_m \nabla f(x_m)\cdot \tilde \nabla f(x_m) 
	+ (L/2) |\alpha_m \tilde \nabla f(x_m)|^2.
	\ea
	Taking expectations on both sides, and applying \eqref{var} and \eqref{enabeq}, we obtain
	\ba\label{sebase}
	E[f(x_{m+1})]- E[f(x_{m})] &\leq -\alpha_m E[|\nabla f(x_m)|^2]
	+ (L/2)\alpha_m^2 E[| f(x_m)|^2] + (L/2)\alpha_m^2 \sigma^2. 
	\ea
	Taking $\alpha_m$ sufficiently small that $(L/2)\alpha_m^2\leq \alpha_m/2$, i.e.,
	$\alpha_m\leq 1/L$, we obtain finally the second inequality of \eqref{keyeq}.
	The first inequality of \eqref{keyeq} follows similarly.

	Summing the left- and righthand sides of \eqref{keyeq} from $m=1$ to $M-1$ and observing that
	the lefthand contribution forms a telescoping sum, we obtain after rearrangement
	$$
	\sum_{m=1}^{M-1} \alpha_m E[|\nabla f(w_m)|^2] \leq 2(E[f(w_{1})]-E[f(w_{M})] 
	+\sigma^2 L \sum_{m=1}^{M-1} \alpha_m^2,
	$$
	yielding \eqref{keysum} by nonnegativity of $f$.
\end{proof}

\subsection{The approximately convex case}\label{s:con}

\begin{proof}[Proof of Proposition \ref{main2}]
	Defining $F_m:= E[f(w_m)]$, we have by \eqref{acon} $F_m \sim E[|\nabla f(w_m)|^2]$, whence,
	substituting into \eqref{keyeq}, we obtain for some $c, C>0$
	$$
	F_{m+1}\leq (1- c\alpha_m)F_m  + C \alpha_m^2.
	$$
	This linear recursive inequality may be solved by discrete variation of constants/Duhamel principle, to
	give
	\be\label{Frep}
	F_{m+1}\leq  F_1 \Pi_{j=1}^{m} (1-c\alpha_j)  + \sum_{i=1}^m C\alpha_i^2 \Pi_{j=i}^{m}(1-c\alpha_j).
	\ee
	Using that $\alpha_m$ is nonincreasing and small, we may estimate the first term on the righthand side
	by
	$$
	F_1 e^{\sum_{j=1}^m \log (1-c\alpha_j)} \sim F_1 e^{-\sum_{j=1}^m c\alpha_j},
	$$
	which, by \eqref{suminf}, goes to zero as $m\to \infty$.

	To estimate the second term, introduce an extension $\alpha(\cdot)$ of $\alpha_j$
	to the positive real line, defined as any $C^1$ nondecreasing function such that $\alpha(j)=\alpha_j$,
	$|\alpha'|\leq K$, and $|\alpha|$ is sufficiently small.
	Using the integral test, and estimating products by exponentials of sums as above, we then have
	$$
	\sum_{i=1}^m C\alpha_i^2 \Pi_{j=i}^{m}(1-c\alpha_j) \sim
	C\int_1^m\alpha(i)^2 e^{-c\int{i}^{m}\alpha(j) dj}di.
	$$
	Using the fact that $ (d/di)e^{-c\int_i^{m}\alpha(j) dj} = c\alpha \e^{-c\int_i^{m}\alpha(j) dj} $,
	we find, integrating by parts, that this may be bounded by
 $$
 \begin{aligned}
	 C\int_1^m\alpha(i)^2 e^{-c\int{i}^{m}\alpha(j) dj}di &=
	 (C/c) \int_1^{m} \alpha(i)(d/di) e^{-c\int_i^{m}\alpha(j) dj} di \\
	 &= (C/c)\alpha(i) e^{-c\int_i^{m}\alpha(j) dj}|_1^m
	 -(C/c)\int_1^{m} \alpha'(j) e^{-c\int_i^{m}\alpha(j) dj}di \\
	 &\leq  (C/c) \Big(\alpha(m) - \alpha(1) e^{-c\int_1^{m}\alpha(j)dj}\Big)\\
	 &\quad
	 - (C/c) \int_{m_0}^m \alpha'(i) di - K(C/c)m_0 e^{-c\int_{m_0}^{m}\alpha(j) dj} \\
 \end{aligned}
 $$
 for any $m_0$, hence converges to zero as $m\to \infty$
 by $\alpha(m)\to 0$ and  $e^{-c\int_{m_0}^{m}\alpha(j) dj} \to 0$.
\end{proof}

\subsection{The general (nonconvex) case}\label{s:gen}
\begin{proof}[Proof of Theorem \ref{main}]
Note, by \eqref{sumfin}, that $\alpha_m\to 0$. Thus, for any $\eps>0$,
we can eventually choose $m_j\to \infty$ such that
\be\label{oeps}
\hbox{\rm $\eps/2< \sum_{m_j+1}^{m_{j+1}} \alpha_m \leq \eps$ for all $j$.}
\ee
	Applying \eqref{keysum} together with \eqref{sumfin}, we obtain
	$$
	\sum_m \alpha_m E[|\nabla f(w_m)|^2]<\infty,
	$$
whence, by the same argument as in the deterministic case, 
$\eps \inf_{m_j<m\leq m_{j+1}}E[|\nabla f(w_m)|^2] \to 0$ as $j\to \infty$, with $\eps>0$ fixed, and thus
eventually
\be\label{sinfcon}
	\hbox{\rm $\inf_{m_j<m\leq m_{j+1}}E[|\nabla f(w_m)|^2]\leq \eps $ for $j$ sufficiently large.}
\ee

On the other hand, for $m, n\in [m_j+1,m_{j+1}]$, $m<n$,
$
|w_n-w_m| \leq \sum_{j=m+1}^{n} \alpha_j |\tilde \nabla f(w_j)|,
$
hence, by \eqref{Hess},
$$
|\nabla f(w_n)-\nabla f(w_m)|^2\leq L^2
\Big(\eps\sum_{j=m+1}^{n} (\alpha_j/\eps) |\tilde \nabla f(w_j)|\Big)^2,
$$
which, by Jenssen's inequality (noting that, by 
\eqref{oeps},
	$\sum_{j=m+1}^{n} (\alpha_j/\eps) |\tilde \nabla f(w_j)|$ is approximately a weighted average of $|f(w_j)|$)
is less than or equal to a bounded multiple of
$$
	L^2 \eps^2 \sum_{j=m+1}^{n} (\alpha_j/\eps)  |\tilde \nabla f(w_j)|^2.
$$

It follows then 
that
$$
\begin{aligned}
	E[|\nabla f(w_n)- \nabla f(w_m)|^2] 
	&\leq 
	2 L^2 \eps^2 \sum_{j=m+1}^{n} (\alpha_j/\eps)  E[|\tilde \nabla f(w_j)|^2]\\
	&\leq
	2 L^2 \eps^2 \sum_{j=m+1}^{n} (\alpha_j/\eps)  \big(E[| \nabla f(w_j)|^2]+ \sigma^2\big) \\
	&=
	2 L^2 \eps \big( \sum_{j=m+1}^{n} \alpha_j  \big(E[| \nabla f(w_j)|^2]\big) + 2L^2  \eps^2 \sigma^2 = O(\eps),
\end{aligned}
$$
	and thus, by the vector inequality $|v|^2\leq 2(|v-w|^2 + |w|^2)$, that
	$$
	E[|\nabla f(w_n)|^2] \leq O(\eps) + 2 E[|\nabla f(w_m)|^2].
	$$

Combining these results, we have that $E[|\nabla f(w_m)|^2]$ is eventually $O(\eps)$, and, as $\eps>0$ was
arbitrary, that $E[|\nabla f(w_m)|^2]\to 0$, as claimed, verifying (ii).
	By Chebyshev's inequality, this gives $|\nabla f(w_m)|\to 0$ in probability, whence
	(iii) then follows as in the deterministic case.

	To verify (i), we have only to sum the three sides of \eqref{keyeq} from $m=M$ to $N-1$, yielding
	by telescoping of the middle sum
	$$
	\begin{aligned}
		-\sum_{m=M}^{N-1} (2\alpha_m)E[|\nabla f(w_m)|^2] - &\sigma^2 (L/2) \sum_{m=M}^{N-1} \alpha_m^2 \leq
	E[f(w_{N})]- E[f(w_{M})]\\
	&\leq
		- \sum_{m=M}^{N-1}(\alpha_m/2)E[|\nabla f(w_m)|^2] + \sigma^2 (L/2) \sum_{m=M}^{N-1}\alpha_m^2.
	\end{aligned}
	$$
	Thus, $ E[f(w_{N})]- E[f(w_{M})]$ is bounded above and below by the tails from $M$ to $N-1$ of
	the sums of two convergent series, hence goes to zero as $M,N\to \infty$.
	The sequence $E[f(w_m)]$ is thus Cauchy, hence converges as $m\to \infty$ to its $\liminf$ and $\limsup$,
	yielding (i).
\end{proof}

\subsection{The time-averaged case}\label{s:timeav}

\begin{proof}[Proof of Proposition \ref{avprop}]
	The time-averaged result \eqref{aprob}-\eqref{averaged} may be obtained directly from
\eqref{suminf}, \eqref{keysum}, by
$$
E[|\nabla f(z_m)|^2]= 
(1/\sum_{j=1}^m \alpha_j) \sum_{j=1}^m \alpha_j E[|\nabla f(w_j)|^2]
< (1/\sum_{j=1}^m \alpha_j)\Big(E[f(w_1)]+ \sigma^2 L \sum_{j=1}^m \alpha_j^2\Big) \to 0,
$$
so long as $ \frac{\sum_{j=1}^M \alpha_j^2 }{\sum_{j=1}^m \alpha_j}\to 0$ as $m\to \infty$,
as holds in particular if $\lim_{m\to \infty}\alpha_m=0$.
\end{proof}

\br\label{comprmk}
Comparing the proof of Proposition \ref{avprop} to that of Theorem \ref{main},
we see that the new element in the latter (i.e., in the nonaveraged case) 
is the same time-batching/Jensen estimate as in the deterministic case.
\er

\subsection{The nonconvex case revisited}\label{s:revisit}
Dropping the square summability condition \eqref{sumfin}, we can still recover for nonconvex functions
the following partial result.

\begin{proposition}\label{rvprop}
	For nonnegative $f\in C^2$ and $\alpha_m$ satisfying \eqref{small}, \eqref{suminf}, \eqref{Hess}
	\eqref{var}, and $\lim_{m\to \infty}\alpha_m=0$,
	and any random variable $\{w_m\}$ satisfying \eqref{SGD} with $E[f(w_m)]<\infty$ for all $m$
	 \be\label{limlim}
	 \liminf_{m\to \infty} E[|\nabla f(w_m)|^2]=0.
	 \ee
\end{proposition}

\begin{proof}
	Suppose by way of contradiction that 
	\be\label{cont}
	\hbox{\rm $ E[|\nabla f(w_m)|^2]\geq \eps>0$ for $m\geq M$, some $\eps>0$.}
	\ee
	By \eqref{small}, \eqref{keyeq}, and $\lim_{m\to \infty}\alpha_m=0$, we have,
	taking $\alpha_m\leq \eps/2\sigma^2L$:
	$$
	\begin{aligned}
	E[f(w_{m+1})]- E[f(w_{m})] &\leq
	- (\alpha_m/2)E[|\nabla f(w_m)|^2] + \sigma^2 (L/2) \alpha_m^2\\
		&\leq - (\alpha_m/4)E[|\nabla f(w_m)|^2]
	\end{aligned}
	$$
	for $m\geq M_2$, some $M_2\geq M$.
	Summing left and right sides, and dropping terms with favorable sign, we obtain therefore
	$$
	\sum_{m=M_2}^\infty (\alpha_m/4)E[|\nabla f(w_m)|^2] \leq E[f(w_{M_2})]<\infty.
	$$
	But, on the other hand, by \eqref{cont} and \eqref{suminf}, we have
	$$
	\sum_{m=M_2}^\infty (\alpha_m/4)E[|\nabla f(w_m)|^2] \geq \eps\sum_{m=M_2}^\infty \alpha_m=\infty,$$
	a contradiction. By contradiction, therefore, \eqref{cont} is false, giving the result.
\end{proof}

\br\label{halfrmk}
The above is roughly half of the alternative proof of Proposition \ref{suffprop}.
\er

The required finiteness of $E[f(w_m)]$ follows, for example, for the canonical problem \eqref{canform} under
auxiliary assumption \eqref{auxcond}, by a.s. boundedness of $f(w_m)$.
The following result gives a much more general condition for finiteness of $E[f(w_m)]$.

\begin{lemma}
	For nonnegative $f\in C^2$ and $\alpha_m$ satisfying \eqref{small}, \eqref{suminf}, \eqref{Hess}
	\eqref{var}, and 
	$$
	\lim_{m\to \infty}\alpha_m=0,
	$$
	and any random variable $\{w_m\}$ satisfying \eqref{SGD} with $E[f(w_1)]<\infty$, 
	we have $E[f(w_m)]<\infty$ for all $m$ under the growth condition	
	\be\label{fgrowth}
	\hbox{ $f(w)\leq C_1 + C_2 |\nabla f(w)|^2$ for some $C_1, C_2>0$.}
	\ee
\end{lemma}

\begin{proof}
	By the previous proof, for any $\eps>0$, $E[f(w_{m+1})] < E[f(w_{m+1})]$ 
	for $\alpha_m$ sufficiently small and $E[|\nabla f(w_m)|^2]>\eps$.
	But, on the other hand, if
	$E[|\nabla f(w_m)|^2] \leq \eps$ then by condition \eqref{fgrowth} we have 
	$E[f(w_m)]\leq C_1 + C_2 \eps$, whereupon, by \eqref{keyeq}, we have
	$$
	E[f(w_{m+1})]\leq E[f(w_m)]+\sigma^2(L/2)\alpha_m^2\leq 
	C_\eps:= C_1 + C_2 \eps + \alpha_m^2\sigma^2 (L/2). 
	$$
	Combining these observations, we find by induction that 
	$$
	\hbox{\rm  $E[f(w_m)]\leq \max \{ E[f(w_1)], C_\eps\}<\infty$ for all $m\geq 1$.}
	$$
\end{proof}

\subsection{Stochastic coordinate descent}\label{s:scd} For (SCD), we may establish a bit more,
arguing essentially as in the deterministic case. In particular, we require only $\alpha_m\ll 1$
for convergence, and not $\alpha_m\to 0$.

\begin{proposition}\label{suffpropscd}
	For $f\in C^2$ and $\alpha_m$ satisfying \eqref{small}, \eqref{suminf}, and \eqref{Hess}, 
	we have for any solution of \eqref{scd} that (i) $E[f(w_m)]$ is monotone decreasing, and 
	(ii) $E[|\nabla f(w_m)|^2]\to 0$ as $m\to\infty$.
	If, also, $|f|\to \infty$ as $|w|\to \infty$, then 
	(iii) $w_m$ converges as $m\to\infty$ to the critical set $\mathcal{C}:=\{w:\, \nabla f(w)=0\}$. 
\end{proposition}

\begin{proof}
	By direct calculation, 
	$$
	E[\nabla f(w_m)\cdot \tilde \nabla f(w_m)]=(1/d) E[|\tilde \nabla f(w_m)|^2]=  E[|\nabla f(w_m)|^2],
	$$
	whence \eqref{keyeq} reduces for $m$ large enough (hence $\alpha_m$ sufficiently small) to
	\be\label{keyeqscd}
	E[f(w_{m+1})]- E[f(w_{m})] \leq - \alpha_m E[|\nabla f(w_m)|^2].
	\ee
	This gives monotone decrease in $E[f(w_{m})]$, verifying (ii),
	and also summability of $\alpha_m E[|\nabla f(w_m)|^2]$.

By $\alpha_m\to 0$, we can eventually choose $m_j\to \infty$ such that
$1/2< \sum_{m_j+1}^{m_{j+1}} \alpha_m \leq 1$ for all $j$, whence 
\be\label{infeq1}
\hbox{\rm $\inf_{m_j<m\leq m_{j+1}}E[|\nabla f(w_m)|^2] \to 0$ as $j\to \infty$}
\ee
and also
\be\label{infeq2}
\hbox{\rm $\sum_{m_j<m\leq m_{j+1}} \alpha_m E[|\nabla f(w_m)|^2] \to 0$ as $j\to \infty$}
\ee

On the other hand, for $m, n\in [m_j+1,m_{j+1}]$, $m<n$,
$
|w_n-w_m| \leq \sum_{j=m+1}^{n} \alpha_j |\tilde \nabla f(w_j)|,
$
hence, by \eqref{Hess},
$$
|\nabla f(w_n)-\nabla f(w_m)|^2\leq L^2 \Big(\sum_{j=m+1}^{n} \alpha_j |\tilde \nabla f(w_j)|\Big)^2,
$$
which, by Jenssen's inequality, is less than or equal to a bounded multiple of
$$
	L^2 \sum_{j=m+1}^{n} \alpha_j  |\tilde \nabla f(w_j)|^2.
$$
It follows that
$$
|\nabla f(w_n)|^2 \leq 2|\nabla f(w_m)|^2 +  L^2 \Big(\sum_{j=m+1}^{n} \alpha_j |\tilde \nabla f(w_j)|\Big)^2,
$$
whence, taking expectations, we have
$$
\begin{aligned}
E[|\nabla f(w_n)|^2] &\leq 
2E[|\nabla f(w_m)|^2] +  L^2 \sum_{j=m+1}^{n} \alpha_j  E[|\tilde \nabla f(w_j)|^2]\\
&\leq 2E[|\nabla f(w_m)|^2] +  d L^2  \sum_{j=m+1}^{n} \alpha_j  E[|\nabla f(w_j)|^2].
\end{aligned}
$$

Taking $m,n\in [m_j,m_{j+1}]$ with the infimum of $E[|\nabla f|^2]$
in $[m_j,m_{j+1}]$ achieved at $m$, and combining \eqref{infeq1} and \eqref{infeq2}, 
we thus obtain $E[|\nabla f(w_n)|^2]\to 0$ as $j\to \infty$, verifying (ii).
The rest goes as in the proof of Theorem \ref{main}.
\end{proof}

\br\label{stochrmk}
Note, in the proof of Proposition \ref{suffpropscd}, that the
size of $\alpha_m$ must be chosen $d$ times smaller than in the treatment of the standard case,
in order that $d\alpha_m$ be dominated by the Hessian bound $L$,
where $d$ is the number of coordinate directions.  
This may not be the rate-determining factor, but anyway somewhat 
nullifies the $d$ times savings in computation afforded by (SCD).
\er

\section{Some simple examples}\label{s:egs}
We illustrate the theory with some low-dimensional examples, based on the concrete form \eqref{canform}:
$$
f(w)=(1/N)\sum_{i=1}^N f_i(w),  
\qquad \tilde f(w):=(1/b) \sum_{i\in S} f_i(w),
$$
$f,f_k: \R^d\to \R$, where 
$S$ is chosen with equal likelihood among $S\subset \{1, \dots, N\}$ of size $|S|=b$.

\subsection{An explicitly solvable case}\label{s:simple}
We start with the simplest case $d=1$, $N=2$, $b=1$, and a convex example that we can essentially 
solve completely.
\be\label{xeq}
f(x)= (1/2)(f_1+f_2)(x)=x^2+1, \qquad f_1(x)=(x-1)^2,\quad f_2(x)=(x+1)^2. 
\ee
Then,
$$
\nabla f(x)=2x,
$$
while 
$ \tilde \nabla f(x)$ is  $2(x-1)$ with probability $1/2$ and $2(x+1)$ with probability $1/2$, or
\be \label{tildenabla} 
 \tilde \nabla f(x)= 2(x+\theta),
 \ee
 where $\theta=\pm 1$ with probability $1/2$.
 Thus, stochastic gradient descent corresponds to the stochastic linear recursion relation
 $x_{m+1}=x_m - 2\alpha_m(x_m + \theta)$, or
 \be\label{egsgd}
 x_{m+1}=(1-2\alpha_m) x_m - 2\alpha_m \theta.
 \ee

 This in turn may be reduced by the variation of constants transformation
 \be\label{vc}
 y_m \Pi_{i=1}^n (1-2\alpha_i)=x_m
 \ee
 to a summation
 $y_m:= y_m- 2 \alpha_m/\Pi_{i=1}^n(1-2\alpha_i)$, or
 \be\label{redsum}
 y_{m+1}-y_1= \sum_{i=1}^{m } X_i, \qquad X_m:= - 2 \alpha_m \theta/\Pi_{i=1}^n(1-2\alpha_i),
 \ee
 that is, a random walk with varying step size.
 This may be expected under suitable conditions to converge to a normal random variable determined by 
 expectation and variance alone.  
 Indeed, necessary and sufficient conditions for asymptotic normality are given in the present case by
 the well-known Lindeberg-Feller theorem.

 \begin{proposition}[Lindeberg-Feller theorem]\label{LF}
Let $X_n$, $n\geq 1$, be independent random variables with finite second moments. 
	 Let $\sigma_n^2={\rm Var} [X_n]$ and $s_n^2=\sum_{i=1}^n \sigma_n^2$.
For $\eps > 0$, set
	 \be\label{sneps}
	 s_{n,\eps}^2:=E[|X_n - E[X_n]|^2;\, |X_n - E[X_n]|^2 \geq \eps s_n]
	 \ee
	 Assume that $\lim_{n\to \infty} s_n =\infty$ and $\lim_{n\to \infty}\sigma_n/s_n=0$.
	 Then, $(1/s_n)\sum_{i=1}^n (X_i-E[x_i])$ 
converges in distribution to the standard normal if and only if for all $\eps>0$,
	 \be\label{LFcond}
	 \lim_{n\to \infty} s_{n,\eps}/s_n=0.
	 \ee
 \end{proposition}

 \bc\label{normalcor}
 For $\alpha_m=\alpha(x)$ satisfying \eqref{suminf} for $\alpha(\cdot)$ continuous and monotone decreasing as
 a function over the reals, $x_m$ is asymptotically normal.
 Furthermore, the rate of convergence is bounded by $\sigma_m/s_n$.
 \ec

 \begin{proof}
 Evidently, ${\rm Var}[\theta]=1$, whence 
 $$
 \sigma_m^2:={\rm Var}[ X_m]= \Big( - 2 \alpha_m /\Pi_{i=1}^n(1-2\alpha_i)\Big)^2
 \sim 4\alpha_m^2 e^{-4\sum_{i=1}^m \alpha_i }\to \infty
 $$
 as $m\to \infty$, by assumption \eqref{suminf}, and thus
 $ s_n^2:=\sum_{i=1}^n \sigma_i^2 \to \infty $ as $m\to \infty$, as well.

	 More precisely, by the integral test, $\sigma_m^2\sim 4\alpha(m)^2 e^{2\int_1^m \alpha(z)dz}$,
	 while 
	 $ s_n^2 \sim \int_0^n 4\alpha(m)^2 e^{4\int_1^m \alpha(z)dz}dm, $
	 hence, noting that 
	 $(d/dm)e^{4\int_1^m \alpha(z)dz}= 4\alpha(m) e^{4\int_1^m \alpha(z)dz}$, and integrating by parts,
	 we obtain
	 $$
	 s_n^2 \sim \alpha(m)^2 e^{4\int_1^m \alpha(z)dz}|_0^n -
	 \int_0^n \alpha'(m) e^{4\int_1^m \alpha(z)dz}dm.
	 $$
	 Using $\alpha'<0$ gives then
	 $$
	 s_n^2 \gtrsim \alpha(m)^2 e^{4\int_1^m \alpha(z)dz}|_0^n \sim \alpha_m^{-1} \sigma_m^2,
	 $$
	 and thus $\lim_{n\to \infty}\sigma_n/s_n\sim \lim_{n\to \infty} \sqrt{\alpha_n} = 0$,
	 verifying the hypotheses of Proposition \ref{LF}.

	 It follows that $y_m=\sum_{n=1}^m X_n$ is asymptotically normal if and only if the Lindeberg-Feller condition
	 \eqref{LFcond} is satisfied. But, noting that in our case $|X_n-E[X_n]|^2= {\bf Var}[X_n]=\sigma_n^2$ 
	 with probability one, we see that \eqref{LFcond} is equivalent to the hypothesis $\sigma_n/s_n\to 0$
	 already verified. This proves that $y_m$ is asymptotically normal, hence $x_m$, being a constant multiple
	 of $y_m$, is asymptotically normal as well.
	 Finally, the rate of convergence follows as in \cite{LZ,LM} by Esseen's theorem \cite[Section XVI.5]{F}
	 since the third moment of any random variable is bounded by the product of its supremum times its second
moment.
 \end{proof}

 Having shown that $x_m$ is asymptotically normal, we need only determine its mean and variance to
 describe its asymptotic behavior.  But, these may be found exactly by deterministic recurrence relations,
 as we now show.

 \begin{proposition}\label{recprop}
	 Under the assumptions of Corollary \ref{normalcor}, we have for expectation $E_m:=E[x_m]$,
	 second moment $F_m:=E[|x_m|^2]$, and variance $V_m:={\rm Var}[x_m]$
	 the recursions
	 \ba\label{recursions}
	 E_{m+1}&=(1-2\alpha_m) E_m,\\
	 F_{m+1}&=(1-2\alpha_m)^2 F_m +  4\alpha_m^2,\\
	 V_{m+1}&=(1-2\alpha_m)^2 V_m +  4\alpha_m^2,
	 \ea
	 giving solution formulae
	 \ba\label{recsolns}
	 E_{m}&= \Pi_{i=1}^{m-1}(1-2\alpha_i) E_1,\\
	 F_{m}&=\Pi_{i=1}^{m-1}(1-2\alpha_i)^2 F_1 + \sum_{j=1}^{m-1}
	 4\alpha_j^2  \Pi_{i=j}^{m-1}  (1-2\alpha_i)^2,\\
	 V_{m}&=\Pi_{i=1}^{m-1}(1-2\alpha_i)^2 V_1 + \sum_{j=1}^{m-1}
	 4\alpha_j^2  \Pi_{i=j}^{m-1}  (1-2\alpha_i)^2.
	 \ea
 \end{proposition}

 \begin{proof}
	 The first assertion follows immediately from \eqref{egsgd} and $E[\theta]=0$.
	 Squaring both sides of \eqref{egsgd} we obtain
	 $$
	 \begin{aligned}
	 x_{m+1}^2-x_m^2&= \Big(x_m(1-2\alpha_m)  -2\alpha_m \theta\Big)^2 - x_m^2\\
		 &= (1-2\alpha_m)^2 x_m^2 - 4(1-2\alpha_m)x_m \alpha_m \theta
		 + 4\alpha_m^2 \theta^2 -x_m^2,
	 \end{aligned}
	$$
	 whence, taking expectations using $E[\theta]=0$, $E[\theta^2]=1$, and rearranging, we obtain
	 the asserted recurrence for $F_m$, and thus, by  $V_m=F_m-E_m^2$, the asserted recurrence for $V_m$.
	 The formulae \eqref{recsolns} then follow by (discrete) variation of constants, similarly
	 as in the proof of Corollary \ref{normalcor}.
 \end{proof}

 \br\label{egrmk}
 From \eqref{recursions}-\eqref{recsolns}, we see that for this example, convergence of the expected value exactly
 follows that of the deterministic case, while accessing half of the data points for each step.
 Meanwhile, the first term of the variance formula gives exactly the contribution of the deterministic case,
 while the second accounts for the contribution of counterbalancing stochastic effects.
 \er

 \br\label{sharprmk}
	The variance formula \eqref{recsolns}(iii), together with the fact that $|\nabla f(x)|\sim |x|^2$,
	shows that \eqref{suminf} and $\lim_{m\to \infty} \alpha_m=0$ are both necessary for
	convergence to zero of $E[|\nabla f(x_m)|]^2$. For, the first term of
	\eqref{recsolns}(iii) does not go to zero unless \eqref{suminf} holds, while the third term
	includes the contribution $4\alpha_{m-1}^2(1-2\alpha_{m-1})\sim \alpha_{m-1}^2$, 
	hence does not go to zero unless $\lim_{m\to \infty} \alpha_m=0$. 
 \er

 \smallskip
 {\bf Case ($\alpha_m=c/m$).}
 Depending on the choice of step size $\alpha_m$, either the deterministic or the stochastic error may dominate.
 For example, in the case $\alpha_m=c/m$, the stochastic part is order 
 $$
 \begin{aligned}
	 \int_1^m \alpha_i^2 e^{-4c\int_{i}^m j^{-1}dj} di&= \int_1^m c^2 i^{-2} (m/i)^{-4c} di\\
	 & = c^2 m^{-4c} \int_1^m  i^{4c-2}  di\\
	 &\sim \begin{cases}
		 m^{-4c}  & c<1/4,\\
		 m^{-1} & c \geq 1/4,
	 \end{cases}
 \end{aligned}
 $$
 while the deterministic part is order $e^{-4c\log m}= m^{-4c}$.
Thus, for $c<1/4$, the stochastic contribution decays at the same rate $m^{-4c}$ as the deterministic part.
For $c>1/4$, the stochastic part decays as slower $m^{-1}$ rate, giving convergence slower than the
$m^{-4c}$ deterministic rate.
Note that we require in general that $\alpha_m$ be small, both in our estimates by Taylor expansion
and in order to avoid overshoot and other undesirable numerical behavior. Indeed, for fixed-step
gradient descent, the step size is taken less than $1/L$, where $L$ is the maximum of the Hessian,
in this case $1/2$, giving $c>1/2$ as a practical upper bound.

 \smallskip
 {\bf Case ($\alpha_m=cm^{-p}$, $1/2<p<1$).}
 In the case $\alpha_m=cm^{-p}$, $1/2<p<1$ on the other hand,
 we have a deterministic decay rate of 
 $$
 e^{-4\sum_{i=1}^{m-1} \alpha_i} \sim e^{-4c\int_1^m i^{-p}di}\sim e^{-4c m^{1-p}}.
 $$
 Meanwhile, the stochastic contribution is asymptotic to
 $$
 \int_1^{m} \alpha_i^2 e^{-4c(m^{1-p}- i^{1-p})} di =
 c^2 e^{-4c m^{1-p}} \int_1^{m} i^{-2p} e^{4ci^{1-p} } di .
 $$
 Noting that 
 $(d/di) e^{4ci^{1-p}}= 4c(1-p)i^{-p}$, we may integrate by parts to obtain
 $$
 c^2 e^{-4c m^{1-p}} \Big(  i^{-p} e^{4ci^{1-p}}|_1^m + p \int_1^{m} i^{-p-1} e^{4ci^{1-p}} di \Big).
 $$
 Observing that the second term on the righthand side is vanishingly small compared to the first,
 by $i^{-2p}\gg i^{-p-1}$, we see that the stochastic portion is asymptotic to
 $$
 c^2 e^{-4c m^{1-p}}   m^{-p} e^{4c m^{1-p} }= c^2 m^{-p}, 
 $$
 algebraic, hence {\it always} much larger than the exponentially decaying
 deterministic part, independent of the value of $c$.

 \smallskip
 {\bf Case ($\alpha_m\equiv \alpha \ll 1$).} Finally, we consider the interesting (nonconvergent) 
 case of constant but small $\alpha_m$, for which \eqref{recursions} become scalar linear
 autonomous discrete dynamical systems, with exponents $1-2\alpha, (1-2\alpha)^2 <1$.
 Evidently, the shift map for variance $V_n$ has a unique attracting fixed point 
 \be\label{fixedpt}
 V_*=\frac{4\alpha^2}{(1-2\alpha)^2},
 \ee
 toward which the variance approaches exponentially with rate $(1-2\alpha)^{2n}$,
 and a limiting standard of deviation $2\alpha/ (1-2\alpha)$.
 That is, for $\alpha \ll 1$, solutions of \eqref{egsgd}, though they do not converge to the minimum $x=0$
 of $f(x)$, {\it do} converge to a neighborhood of order $\alpha$ of $x=0$.
 In practice this may be quite satisfactory and for its simplicity this choice is often used.

 \smallskip
 \textbf{Numerical comparison.} In Figure \ref{ConvexEmpericalVSPredictedMonteCarlo}, we compare the empirical distribution obtained by
 Monte Carlo simulation starting with a fixed $x_1$
 with the normal distribution of expectation and variance determined by \eqref{recursions},
 obtaining excellent correspondence.

\begin{figure}
    \centering
    \textbf{Numerics for Convex Example}\par\medskip
    \includegraphics[scale=.6]{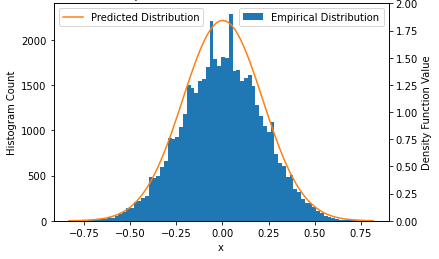}
    \caption{\small
		A histogram of the results of a Monte-Carlo simulation of SGD for the convex function \eqref{xeq} using $50,000$ trials plotted simultaneously with the density function predicted by \eqref{recursions}.
    }
    \label{ConvexEmpericalVSPredictedMonteCarlo}
\end{figure}

\subsection{Nonconvex case}\label{s:nonconvex}
Taking again $d=b=1$, $N=2$, we may take $f(x)$ to be any nonconvex function, for example
\ba\label{ncf}
f(x)&= x^4+3x^3-4x+2,\\
f_1(x)&=f(x)-g(x),\\
f_2(x)&=f(x)+g(x). 
\ea
For the simplest case $g(x)= \sigma x$ this yields iteration
\ba\label{noncon}
x_{m+1}&=x_m - \alpha_m f'(x_m) + \alpha_m \sigma \theta\\
&= x_m - \alpha_m ( f'(x_m) + \sigma \theta),
\ea
where $\theta=\pm 1$ with probability $1/2$.
It is no longer linear, so does not admit an explicit variation of constants solution;
nor can it be expected to yield an asymptotically normal distribution.
And, indeed, Monte Carlo simulation starting with a fixed $x_1$, as depicted in Figure \ref{fig:nonconvexSGD} below,
yields an empirical distribution far from normal, as we see by comparison with a normal distribution 
of equal expectation and variance.

\begin{figure}
    \centering
    \textbf{Empirical vs Normal Distribution for Nonconvex Function}\par\medskip
    \begin{subfigure}[b]{0.475\textwidth}
        \centering
        \includegraphics[width=\textwidth]{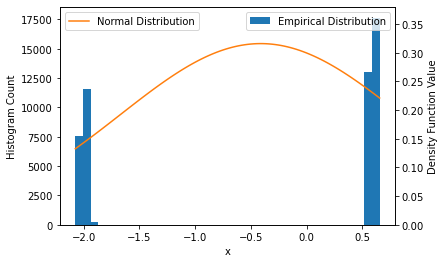}
        \caption[]%
        {{\small $\sigma=1$, $c=0.1$}}    
        \label{fig:c1s1}
    \end{subfigure}
    \hfill
    \begin{subfigure}[b]{0.475\textwidth}  
        \centering 
        \includegraphics[width=\textwidth]{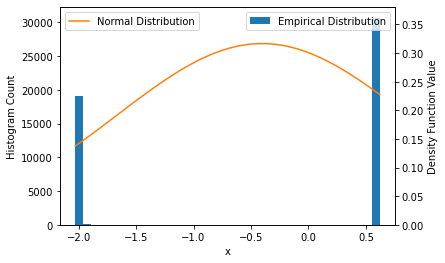}
        \caption[]%
        {{\small $\sigma=1$, $c=0.01$}}    
        \label{fig:c01s1}
    \end{subfigure}
    \vskip\baselineskip
    \begin{subfigure}[b]{0.475\textwidth}   
        \centering 
        \includegraphics[width=\textwidth]{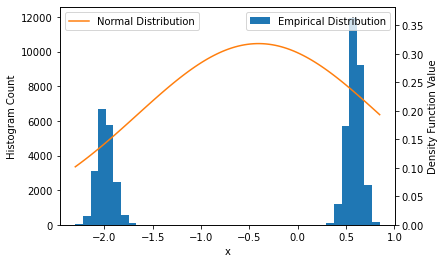}
        \caption[]%
        {{\small  $\sigma=10$, $c=0.01$}}    
        \label{fig:c01s10}
    \end{subfigure}
    \hfill
    \begin{subfigure}[b]{0.475\textwidth}   
        \centering 
        \includegraphics[width=\textwidth]{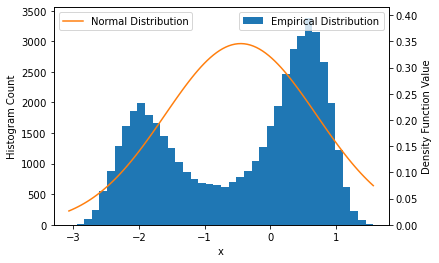}
        \caption[]%
        {{\small $\sigma=50$, $c=0.01$}}    
        \label{fig:c01s50}
    \end{subfigure}
    \caption[]
    {\small Several histograms comparing Monte-Carlo simulations of SGD for the nonconvex function \eqref{ncf}, varying key parameters.  Each uses $50,000$ trials of $500$ SGD iterations plotted simultaneously with a normal distribution of equal mean and variance.  The value of $\sigma$ determines the function $g(x)$, which varies across the figures as labeled.  For all figures $\alpha(m)=\frac{c}{\log(1+m)}$ is used for the stepsize, with variable values of $c$.} 
    \label{fig:nonconvexSGD}
\end{figure}

For reference, we display in Figure \ref{NonconvexGraph} a graph of the function $f(x)$, indicating clearly the
presence of two local minima. Intuitively, some portion of the probability distribution will be trapped near one
local minimum and the rest near the other, hence bivariate rather than normal.

\begin{figure}
    \centering
    \textbf{Nonconvex Function Example}\par\medskip
    \includegraphics[scale=.6]{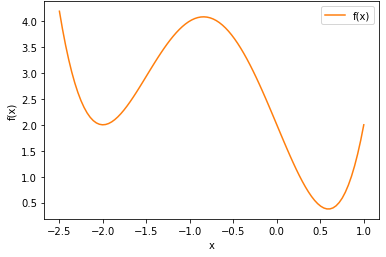}
    \caption{\small
	A plot of the nonconvex function $f(x)$ defined by \eqref{ncf}.
    }
    \label{NonconvexGraph}
\end{figure}

\smallskip
\textbf{The square summability condition.} We now investigate numerically the square summability condition 
\eqref{sumfin} and its relation to convergence in the nonconvex case,
considering the problem \eqref{ncf} and the associated (SGD) scheme \eqref{noncon} for various
choices of $\sigma$ and $\alpha_m$.

\medskip
{\bf Small $\sigma$.} Before beginning, we first make some easy observations about
the small-$\sigma$ case, in which the $\nabla f$ part of \eqref{noncon} dominates the
stochastic term.  Since critical points are nondegenerate,
trajectories are thus clearly trapped- deterministically, that is, for any path-realization,
not probabilistically- in uncertainty balls of radius $O(\sigma)$ around local minima 
if they ever enter, and blocked out of $O(\sigma)$ radius ball around the local maximum if they ever leave,
since $\nabla f$ dominates $O(\sigma)$ corrector in this case. Inside the local max ball, meanwhile, $\nabla f$
is pushing out, so the process is at least as likely to leave as in a standard unbiased random walk with step sizes $\alpha_m$,
so eventually leaves almost surely. Similarly, in the invariant set outside the repelling ball around the local
max, the process is at least as likely to reach a minumum ball as is the unbiased walk, and so it does so almost
surely.  The end result is thus, provably, convergence to support within the two uncertainty balls around the
local minima: a bimodal distribution.  
Eventually, therefore, we reduce for the parts of the solution trapped in uncertainty balls to the 
approximately convex case treated in Proposition \ref{main2}, leading to the complete conclusion \eqref{concon}
therein.  
It is an interesting further question whether the trapped local minimum distributions behave approximately 
as in the quadratic case, say, asymptotically in $\sigma$ as $\sigma\to 0$.

\br\label{bdrmkk}
The above analysis applies to general systems for which we can deterministically bound the stochastic part from
above by a small scaling constant. But, it is mainly theoretical, as this seems unlikely to occur for the canonical
example of least squares/loss functions. 
\er

\medskip
{\bf Large $\sigma$.} In this case, things are far from clear, as individual steps are dominated by the stochastic
random walk part of \eqref{noncon}.  
Indeed, it sheds a nice light on the strength of the previous results on sufficiency, which seem hardly
obvious even for the square summable case.
The results displayed in Figure \ref{fig:nonconvexSGD}(c)-(d) for the large values $\sigma=10, 50$ 
suggest convergence for the very slowly-decaying choice $\alpha_m=c/\log(1+m)$.
Indeed, all the experiments of this paper suggest convergence also for constant $\alpha_m\equiv c \ll 1$,
in this case ``uncertainty balls'' centered at critical points and not critical points themselves,
with radius going to zero as $c\to 0$.

\subsection{Higher-dimensional generalizations}\label{s:hdgen}
For $N=2$, $b=1$, one can take $f:\R^d\to \R$ quadratic and $f_1=f+g$, $f_2=f-g$ with $g(x)=c\cdot x$ linear to
obtain again a linear, hence solvable by variation of constants, stochastic iteration
\be\label{gsolve}
x_{m+1}=x_m - \alpha_m \nabla f + \theta c,
\ee
where $x$, $c$ are vectorial and $\theta$ is $\pm 1$ with probability $1/2$.
Note that the requirement $f\geq 0$ imposes convexity as in the single-variable case $d=1$.
The form is rather special, since stochastic effects are unidirectional in direction $c$ alone.
Thus, starting with a Dirac mass,
one may conclude asymptotic (2d) normality, but with variation in the $c$ direction only.

This low-dimensional artifact may be remedied by taking $N\geq d+1$, so that generically variation will be full rank.
For general $N$, $d$, $f$ quadratic and all $g_i:=f-f_i=c_ix+d_i$ linear, $\sum_{i=1}^Nc_i=\sum_{i=1}^Nd_i=0$,
we obtain a similar form
$ x_{m+1}=x_m - \alpha_m \nabla f + \sum_{i=1}^N \theta_i c_i, $
where vector $\theta$ takes values $e_j$ (standard basis elements) with equal likelihood
$P(\theta=e_j)= 1/N$.
This in turn may be reduced by variation of constants to a matrix-valued variable coefficient
random walk in directions $c_1,\dots, c_N$, where, recall, $C_N=1- \sum_{i=1}^{N-1}c_i$.
Thus, for $N\geq d+1$, generically, $\Span\{c_1, \dots, C_{N-1}\}=\R^d$, 
and so stochastic effects correspond to a nondegenerate diffusion.
Other than this latter effect, the sizes of $N$ and $b$ seem qualitatively not so important, as for $\alpha_m\ll 1$
the Law of Large Numbers should give aggregate short-time stochastic behavior that is approximately normal 
in any case.

 \section{Continuous-time analog and Fokker-Planck approximation}\label{s:FP}
 More generally, the multi-d recursion
 \be\label{md}
 x_{m+1}=x_m - \alpha_m \nabla f + \alpha_m \sum_{i=1}^N \theta_i \nabla f_i(x), 
 \ee
$N$ arbitrary, $x\in \R^d$, $b=1$, and its generalization to $b\geq 1$ 
suggest in the small step size limit $\alpha_m\to 0$
\be\label{stoch}
\dot x(t) = - \alpha(t)u(x(t)) 
+ \alpha(t) \sigma(x) dW_t,  \qquad \sigma=\Sigma^{1/2},
\ee
where $u(x):=\nabla f(x)$, $\Sigma(x)$ is the (symmetric positive semidefinite) covariance matrix of 
$$
\sum_{i=1}^N \theta_i \nabla f_i(x),
$$
and $W_t$ is $d$-dimensional Brownian motion.
We shall not attempt to prove such a result, but only consider it as a heuristic analog, similarly
as we did in Section \ref{s:cont} for the deterministic case.

\subsection{Fokker-Planck approximation}\label{s:FPapprox}
The evolution of the probability density $\rho(x)$ of a solution $X(t)$ of stochastic process \eqref{stoch}
is governed \cite{F,P,K} by the {\it Fokker--Planck equation}
\be\label{FPeq}
\rho_t -  \sum_i \partial_{x_i}( \alpha(t) u_i(x)\rho)= (1/2) \sum_{i,j} (\alpha(t) \Sigma_{ij}(x) \rho), 
\qquad \rho\in \R, \; x, u\in \R^d, \; \Sigma\in \R^{d\times d}.
\ee
In the deterministic case $\Sigma\equiv 0$, this reduces simply to conservation of probability under
convection by vector field $-\alpha(t) u(x)$.

Viewing \eqref{stoch} as a qualitative approximation of \eqref{md}, we 
thus obtain \eqref{FPeq} as a qualitative approximation of the evolution of the probability density associated
with \eqref{md}.
This gives us another approach besides Monte Carlo for numerical investigation of behavior of \eqref{SGD},
namely, numerical approximation of the solution to convection-diffusion equation \eqref{FPeq}, as we now describe.

\subsubsection{Relations to Physics, and solution in a simple case}\label{s:physics}
For the simple example \eqref{xeq}, for which \eqref{stoch} corresponds to a diffusive harmonic
oscillator equation, \eqref{FPeq} reduces to
\be\label{newx}
\rho_t + \alpha(t) (2x \rho)_x= 2 \alpha(t)^2 (\rho)_{xx},
\ee
which may be solved exactly by essentially the same variation of constants coordinate change as in 
Section \ref{s:egs}, namely, $x= e^{-\int_0^t 2\alpha(s)ds} y$, 
$\rho(x)=u(y)(dy/dx)= u(y) e^{\int_0^t 2\alpha(s)ds}$, giving 
$$
u_t = \alpha(t)^2 e^{\int_0^t 4\alpha(s)ds} u_{yy}.
$$
This in turn may be reduced to the heat equation by the change of variable $t\to \tau$,
where 
$$
d\tau/dt= \alpha(t)^2 e^{\int_0^t 4\alpha(s)ds}.
$$

Note that conversion to the heat equation by time-dependent coordinate changes, which preserve the property of
normality, under appropriate assumptions on the initial density $\rho|_{t=0}$ 
yields the result found by Lindeberg-Feller theorem in the discrete case, of convergence 
toward a Gaussian distribution, by heat equation properties.
Moreover, integration of \eqref{newx} against $x$, followed by integration by parts,
yields for expectation $E(t):=\int_{-\infty}^{\infty}x\rho(x)dx$
the ODE $E'(t)= -2\alpha(t)E(t)$ as in the deterministic case $dX= -2\alpha(t) X$,
while integration against $x^2$ after integration by parts
yields for second moment $F(t):= \int_{-\infty}^{\infty}x^2\rho(x)dx$ the ODE 
\be\label{Fode}
F'(t)= -4\alpha(t) F(t)+ 4 \alpha(t)^2,
\ee
completing the description of asymptotic behavior (cf. \eqref{recursions}.

In the special case $\alpha(t)\equiv 1$, \eqref{xeq}, \eqref{stoch} reduces to the Ornstein--Uhlenbeck process
with applications in Brownian motion \cite{UO} and mathematical finance \cite{B,LL}.
This arises also in the study of convergence of solutions of the heat equation to scale-invariant flow.
In this case, \eqref{Fode} yields convergence of $F$ to a limiting equilibrium value $F_*=\alpha$,
similarly as in the discrete case.
Note that this, together with the above observations regarding convergence to normality recover the 
classical phenomenon of convergence of solutions $\rho$ to an explicit equilibrium Maxwellian distribution.

More generally, \eqref{stoch} with $\alpha\equiv 1$ describes Brownian dynamics on an arbitrary energy 
landscape.
See \cite{dMRV} for related discussions of equilibrium measures on general setting.

\subsection{Numerical approximation}\label{s:num}
To simulate the described Fokker-Planck equations, we use a Crank-Nicholson scheme with adaptive upwinding.  Upwinding is a common practice for transport type equations, where gradients are calculated with a bias for the direction the transport is coming from.  This allows for greatly increased stability.  Here we do not know the direction of motion a priori, so we determine the direction with a function $\beta$ of the gradient.  For our one dimensional example, we simply use $\beta(\nabla F) = \sgn(\nabla F)$, though this can easily be generalized for higher dimensions.  The precise method is described in \cite{CK}.

\begin{figure}
        \centering
        \textbf{Fokker-Planck Numerical Simulation}\par\medskip
        \begin{subfigure}[b]{0.475\textwidth}
            \centering
            \includegraphics[width=\textwidth]{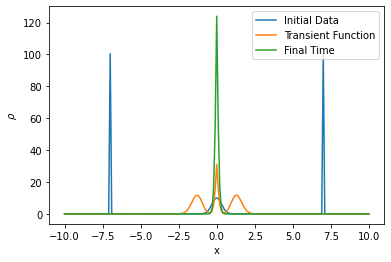}
            \caption[]%
            {{\small Simulation of the convex function \eqref{xeq}}}    
            \label{fig:FP_convex}
        \end{subfigure}
        \hfill
        \begin{subfigure}[b]{0.475\textwidth}  
            \centering 
            \includegraphics[width=\textwidth]{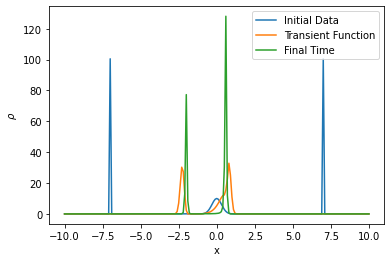}
            \caption[]%
            {{\small Simulation of the nonconvex function \eqref{ncf}}}    
            \label{fig:FP_nonconvex}
        \end{subfigure}
        \caption[]
        {\small Three time slices of a simulation of the Fokker-Planck equations \eqref{newx} corresponding to the previous examples of convex and nonconvex functions.  We see the qualitative agreement of the method with previous Monte-Carlo and analytic results.  For both the initial condition is given by a small Gaussian centered at 0 and an approximation of a point mass far from the center.} 
        \label{fig:FP}
    \end{figure}

\section{Application to 2- and many-player games}\label{s:games}
In this final section we describe a novel application to 2-player games and multi-player games involving
two (asynchronous) coalitions: more generally, to arbitrary problems of form
\be\label{maxprob}
\min \phi(y), \qquad \phi(y) := \max_{1\leq j\leq N}\{ \phi_j(y)\},
\ee
without loss of generality $\phi_j\geq \eta>0$ for all $j$.

The idea is to first approximate the maximum on the righthand side of \eqref{maxprob}
following \cite{BBDJZ} by the smoothed version
$\min \phi_p(y):= \Big(\sum_{1\leq j\leq N} \phi_j(y)^p \Big)^{1/p}$
given by the $\ell^p$ norm of $\{\phi_j\}$.
Defining $\Phi_p:=(1/N)\phi_p^p$, we then convert to the equivalent problem
\be\label{canphi}
\min \Phi_p(y):= (1/N) \sum_{1\leq j\leq N} \phi_j(y)^p
\ee
of canonical form \eqref{canform}, to which (SGD) may be readily applied.
For large problems, $N\gg 1$, the hope is that this will lead to substantial savings in computation time.

\subsection{2-player games}\label{s:2player}
Now, consider a strictly positive $M\times N$ 2-player game with payoff function $\Psi(x,y)=x^TAy$, 
where $x\in M$ and $y\in N$ are probability vectors representing random strategies for players one and two and 
$A$ is an $M\times N$ matrix with entries $A_{ij}\geq \eta>0$.
Following von Neumann's fundamental theorem of games \cite{vN}, the optimal (random) strategy for player two is
\be\label{cang}
\min_y \max_j A_j y,
\ee
where $A_j$ denotes the $j$th row of $A$.
Setting 
\be\label{ydef}
y=(\tilde y, y_N):=(y_1, \dots, y_{N-1}, 1-\sum_{j=1}^{N-1}y_j),
\ee
we find that \eqref{cang} may be expressed as a problem
\be\label{tcang}
\min_{\tilde y} \max_j \phi_j(\tilde y),\qquad \phi_j(\tilde y):= A_j y
\ee
of the form \eqref{maxprob}, with $\tilde y$ varying over the $(N-1)$-simplex $0\leq \tilde y_j\leq 1$,
$\sum \tilde y_j\leq 1$.

An important property of the original 2-player game problem \eqref{cang} is that it is {\it convex};
the following elementary observation shows that, for {\it even} smooothing exponents $p$, 
this important feature is inherited also in the smoothed, approximate problem version \eqref{canphi}
of \eqref{tcang}.

\begin{lemma}\label{smoothlem}
	For $p$ a positive even integer, the smoothed problem \eqref{canphi} associated with \eqref{tcang}
	is \emph{convex} on $\R^{N-1}$, as are, indeed, each of the individual terms $\phi_j(\tilde y)^p$.
	For $p$ an odd positive integer, \eqref{canphi} is convex where $B_j \tilde y\geq 0$ for all $j$,
	or equivalently $A_j\tilde y \geq 0$:
	in particular, on the feasible set $0\leq \tilde y_j\leq 1$, $\sum_j \tilde y_j\leq 1$.
\end{lemma}

\begin{proof}
	Using \eqref{ydef}, we may express $\phi_j(\tilde y)=A_j y$ as $B_j \tilde y$, with $B_j$ constant. 
	Thus, $\phi_j(\tilde y)^p=(B_j \tilde y)^p$, giving
	$
	\nabla_{\tilde y}\phi_j= p(B_j \tilde y)^{p-1} B_j^T
	$
	and $ \nabla^2_{\tilde y}\phi_j= p(p-1)(B_j \tilde y)^{p-2} B_j^T B_j$.
	The latter is evidently positive semidefinite when either $p$ is even or $B_j \tilde y$ is nonnegative.
\end{proof}

Our basic theory then yields convergence of (SGD) for problem \eqref{tcang} for $p$ even and 
any monotone decreasing sequence of step sizes $\alpha_m> 0$ satisfying 
$\sum \alpha_j=\infty$, $\lim_{m\to \infty}\alpha_m=0$.
Here, we have ignored the constraint that $\tilde y$ lie in the feasible set $0\leq \tilde  y_j\leq 1$.
That simplification is valid if the original problem \eqref{cang} has an interior minimum in the feasible
set, since the minima for \eqref{tcang} lie near those of \eqref{cang} for $p$ sufficiently large.
However, in general, one must add a penalty function or other such modification to ensure
that iterates respect the feasibility conditions.

\begin{example}\label{geg}
	Taking $A=\bp 1&3\\2&1\ep$ and $y=(y, 1-y)$, we get 
	$\phi_1(y)= (y+3(1-y))= 3-2y$ and $\phi_2(y)= 2y +(1-y)= y+1$.
	For $p=2$, this gives the minimization problem $ \min_y \Phi_2(y)$ with 
	$$
	\Phi_2(y):= (1/2)(\phi_1(y)^2 + \phi_2(y)^2)= (1/2)(5y^2-10 y +10),
	$$
	and $\nabla \Phi_2= 5y-5$, giving a minimum $\nabla \Phi_2=0$ at $y=1$.
	The associated (SGD) scheme is
	\be\label{gSVD}
	y_{m+1}- y_m= -\alpha_m (5y_m-5) + \theta \alpha_m (3y_m -7),
	\ee
	where $\theta=\pm 1$ with equal probability $1/2$, convergent under our general theory for (SGD).
	By contrast, the solution of the original game problem \eqref{cang} satisfies $(A_1-A_2)(y, 1-y)^T=0$,
	or $2-3y=0$, giving $y=2/3$.
	More generally, one may check that the approximate minima given by different choices of $p$ approach the exact
	value $2/3$ with $O(1/p)$ relative error, i.e., around $\log_{10}p$ digits precision, 
	in keeping with the worst-case error 
	\be\label{wc}
	\|x\|_{\ell^\infty}\leq \|x\|_{\ell^p}\leq \|x\|_{\ell^\infty} e^{\log (n)/p}=
	\|x\|_{\ell^\infty} (1+ O(\log (n)/p))
	\ee
	for $x\in \R^n$, achieved for $|x_j|\equiv \|x\|_{\ell^\infty}$.
	For example, for $p=10$, the minimum of the smoothed function is achieved at $y\approx 0.71$,
	giving relative error of approximately $.05/0.6\approx 0.84$.
\end{example}

\begin{example}\label{RPSeg}
	Taking $A=\bp 2 & 1 & 3\\3& 2&1\\ 1&3&2\ep$ and $y=(y_1,y_2, 1-y_1-y_2)$, $\tilde y=(y_1,y_2)$,
	we get 
	$$
	\begin{aligned}
		\phi_1(\tilde y)&= 2y_1 +y_2 + 3(1-y_1-y_2))= 3-y_1-2y_2, \\ 
		\phi_2(\tilde y)&= 3y_1 + 2 y_2 + (1-y_1-y_2))= 1 + 2y_1+ y_2, \\ 
		\phi_3(\tilde y)&= y_1 + 3 y_2 + 2(1-y_1-y_2))= 2-y_1+ y_2. 
	\end{aligned}
	$$
	This may be recognized as the classical Rock-Paper-Scissors game with payoff boosted by $+2$
	to ensure positivity, with exact optimal strategy $y_1=y_2=1/3$ returning a value of $+2$.
	The associated (SGD) scheme is
	\be\label{gSVD2}
	\tilde y_{m+1}- \tilde y_m= - \sum_j \theta_j \alpha_m p \phi_j(\tilde y)^{p-1} \nabla_{\tilde y}\phi_j,
	\ee
	where $\theta=(1,0,0)$, $(0,1,0)$, or $(0,0,1)$ with equal probability $1/3$, 
	convergent under our general theory for (SGD) to the minimizer of 
	the smoothed, approximate problem \eqref{canphi}, with value $(3 \Phi_p)^{1/p}$ lying according to
	\eqref{wc} within error $\sim 1/p$ of the exact value $+2$.
\end{example}

\medskip

{\bf Issues.} The error bound \eqref{wc} is problematic, as the large $p$ necessary for
accuracy introduces large variations in $\phi_j^p$. This could perhaps be remedied by a multigrid approach,
increasing $p$ as successive iterations (presumably) shrink the computational domain.
Without some such modification, it is not clear whether this approach represents a potential tool
for practical application.

\subsection{Asynchronous coalitions in multi-player games}\label{s:nplayer}
We note that the same method can apply to the $(n-1)$ vs. $1$ ``asynchronized coalition game''\footnote{
	So called because players $2$-$n$ are allowed to coordinate their choices of mixed, or random, strategies,
	but not to synchronize these choices on any single round of play. See \cite{BBDJZ} for further
	discussion.}
\be\label{asynch}
\min_{y_1,\dots, y_{n-1}} \max_j \psi_j(y_1,\dots, y_{n-1})
\ee
studied in \cite{BBDJZ}, where $\psi(y_1,\dots, y_{n-1},x)$ is a multilinear payoff
function and 
$$
\psi_j=\psi_j(y_1,\dots, y_{n-1}, e_j),
$$
of which the simplest ($n=3$) version is
\be\label{aseg}
\min_{y_1,y_2} \max_j\{ y_1^t B_j y_2 \}, \qquad B_j\in \R^{N\times N}.
\ee
Unlike the 2-player version, this is in general a {\it nonconvex optimization problem}
with no relation to linear programming or other standard structures other than the form 
\eqref{maxprob} above.
That is, \eqref{maxprob} isolates 
the most primitive 
property associated with origins from a multiplayer game.

\begin{example}\label{randeg}
	Taking $N=2$, $B_1=\bp 2 & 1 \\3&2\ep$, $B_2=\bp 2 & 5 \\1&2\ep$ and $y_1=(w,1-w)$,
	$y_2=(z,1-z)$, the problem \eqref{aseg} becomes
\be\label{aseg2}
	\min_{w,z} \max_j\{ (w,1-w) B_j (z,1-z)^R \}= \min_{w,z} \max\{ 2+z-w, 2-2wz+3w-z\},
\ee
	or 
\be\label{phis}
\phi_1(w,z)= 2+z-w, \qquad \phi_2(w,z)= 2-2wz+3w-z,
\ee
$0\leq w, \, z \leq 1$, from which we obtain the (SGD) scheme
\ba\label{sd2}
	\bp w_{m+1}\\ z_{m+1}\ep&= \bp w_{m}\\ z_{m}\ep -\alpha_m \theta_1 p(\phi_1^{p-1} \nabla \phi_1) (w_m,z_m)  
	-\alpha_m \theta_2 p(\phi_2^{p-1}\nabla \phi_2)(w_m,z_m) \\
	&=
	\bp w_{m}\\ z_{m}\ep -\alpha_m \theta_1 p\phi_1(w_m,z_m)^{p-1} \bp -1\\1 \ep
	-\alpha_m \theta_2 p\phi_2(w_m,z_m)^{p-1} \bp 3-2z_m\\ -2w_m-1 \ep ,
\ea
where $\theta=(\theta_1,\theta_2)$ is equal to $(1,0)$ or $(0,1)$ with equal probability $1/2$.
	The exact problem \eqref{aseg2} may be seen to be minimized on the boundary.
	For, examining the curve $z=3w/(w+1)$ where $ 2+z-w = 2-2wz+3w-z$, 
and minimizing $2+z-w= 5- 3/(1+w)- w$, we find a unique
	interior critical point at $w=\sqrt{3}-1\approx .73$, which is a maximum.
	The exact minimizer is thus found on the boundary of the domain, where it is readily seen to
	occur at $(w,z)=(0,0)\, , (1,1)$ with value $2$. 
\end{example}

To handle boundary minima as in the above example, we suggest addition of a penalty function, for example
in the present case
\be \label{penalty}
\psi(w,z):= K (w^-)^d + K (z^-)^d + K ((w-1)^+)^d + K ((z-1)^+)^d
\ee
with $K>0$ sufficiently large. 

\br\label{temp} 
Example \eqref{randeg} has some interesting features. The first is that even though $\phi_j(y,z)>0$
for probability vectors, this does not necessarily hold for general $y,z$, and so 
$\sum_j \phi_j^{p}$ is nonnegative for $p$ even, but not necessarily positive.
Indeed, in the present case $\phi_1=\phi_2=0$ is evidently achieved at $w=z+2$ ($\phi_1=0$) and
$0= 2-2(z+2)z+3(z+2)-z= 8 -2z^2 -2z$ ($\phi_2=0$), or 
$ (w,z)=(1/2)(3+\sqrt{17}, -1 + \sqrt{17}) \approx (-.56, -2.56)$,
$(1/2)(3-\sqrt{17}, -1 - \sqrt{17})\approx (3.56, 1.56).  $
These are the unique global minima of $\phi_1^p+\phi_2^p$ for any $p$, lying outside the feasible region.
The second is that the minimax problem \eqref{aseg2} has minima $(w,z)=(0,0), \, (1,1)$ occurring on 
the boundary of the feasible set, which are neither
critical points nor local minima of the extended problem on the plane.
This explains the numerical results of convergence to points outside the feasible region.
\er

\begin{example}\label{randeg2}
	Taking $N=2$, $B_1=\bp 2 & 1 \\3&2\ep$, $B_2=\bp 2 & 2.5 \\1&2\ep$ and $y_1=(w,1-w)$,
	$y_2=(z,1-z)$, the problem \eqref{aseg} becomes
\be\label{aseg3}
	\min_{w,z} \max_j\{ (w,1-w) B_j (z,1-z)^R \}= \min_{w,z} \max\{ 2+z-w, 2+(1/2)wz+(1/2) w-z\},
\ee
	or 
\be\label{phis2}
	\phi_1(w,z)= 2+z-w, \qquad \phi_2(w,z)= 2+(1/2)wz+(1/2)w-z,
\ee
$0\leq w, \, z \leq 1$, from which we obtain the (SGD) scheme
\ba\label{sd3}
	\bp w_{m+1}\\ z_{m+1}\ep&= \bp w_{m}\\ z_{m}\ep -\alpha_m \theta_1 p(\phi_1^{p-1} \nabla \phi_1) (w_m,z_m)  
	-\alpha_m \theta_2 p(\phi_2^{p-1}\nabla \phi_2)(w_m,z_m) ,
\ea
where $\theta=(\theta_1,\theta_2)$ is equal to $(1,0)$ or $(0,1)$ with equal probability $1/2$.
	The exact problem \eqref{aseg3} is readily seen to have minimum on the curve $z=3w/(4-w)$
	where $ 2+z-w = 2+(1/2)wz+(1/2) w-z$. For, on this curve, minimizing $2+z-w= -1+ 12/(4-w)- w$, 
	we find a unique interior critical point at $w=4- \sqrt{12}\approx .536$, with positive
	second derivative, hence a minimum on the dividing curve.
	Meanwhile, on the boundaries of the domain, the minimum value $+2$ is seen to be achieved on the dividing
	curve at $(w,z)=(0,0)$. Thus, the exact minimizer occurs in the interior, on the dividing curve,
	at $(w,z)=(w, 3w/(4-w))\approx (.536,.47)$, with value approximately $1.94$.
Note that the objective function still admits zeros at $(w,z)=(0,2)$ and $(w,z)= (-1,1)$, curiously, but
	starting in the feasible region we do not seem to reach these in 
	numerical experiments.
\end{example}

\begin{example}\label{RPS3}
An interesting $3\times 3\times 3$ case is the ``odd-man-in'' three-player Rock-Paper-Scissors
considered in \cite{BBDJZ}, which has payoff function
\be\label{RPSrptOMI}
\Psi(x,y,z)=2 y\cdot z - x\cdot(y+z),
\ee
where $x,y,z\in \R^3$ are probability vectors.
Considered as an asynchronous coalition game \eqref{aseg} of players $y$, $z$ vs. player $x$,
this has global minimizers at $y=(1,0,0)$, $z=(0,1/2,1/2)$, $y=(0,1,0)$, $z=(1/2,0,1/2)$, 
and $y=(0,0,1)$, $z=(1/2,1/2,0)$,
and local minimizers at $y=(0,1/3,2/3)$, $z=(2/3,1/3,0)$; $y=(2/3,1/3,0)$, $z=(0,1/3,2/3)$; 
	$y=(1/3,0, 2/3)$, $z=(1/3,2/3,0)$; 
$y=(1/3,2/3, 0)$, $z=(1/3,0, 2/3)$; $y=(0,2/3, 1/3)$, $z=(2/3, 0, 1/3)$; and $y=(2/3,0,1/3)$, $z=(0, 2/3, 1/3)$,
with, in addition, a nonsmooth saddle at the Nash equilibrium $y=z=(1/3,1/3,1/3)$.
Adding $2$ to achieve positivity yields the standard form $\min_{y,z}\max \phi_j(y,z)$, where
\be\label{RPSphis}
\phi_j(y,z)= 2+ y\cdot z - (y_j+z_j), \quad j=1,\dots,3,
\ee
with $y=(y_1, y_2, 1-y_1-y_2)$, $z=(z_1, z_2, 1-z_1-z_2)$, 
and $(y_1,y_2)$ lying in the simplices $0\leq y_j, z_j\leq 1$, $y_1+y_2, \, z_1+z_2 \leq 1$,
that is, a minimax problem in $\R^4$.

An interesting reduced problem in $\R^3$ is obtained by restricting $y_1=0$, i.e., ignoring one strategy
option for player $y$, yielding a minimax problem on $0\leq y_2\leq 1$ crossed with the simplex
$0\leq  z_1, z_2\leq 1$, $z_1+z_2 \leq 1$. This inherits the global minimizers $(y_2,z_1,z_2)=(1,1/2,0)$ 
and $(y_2,z_1,z_2)=(0, 1/2,1/2)$, and local minimizers $(y_2, z_1,z_2)=(1/3, 2/3,1/3)$
and $(y_2,z_1,z_2)=(2/3, 2/3, 0)$, 
along with possible new interesting features coming from the restriction to a smaller domain. 
\end{example}

\subsection{Numerical experiments}\label{s:gexp}
We now describe numerical experiments for the examples in Section 6.1 and 6.2.  

\medskip
\textbf{Observations about the importance of stepsize.}
We observe in all of the following examples the importance of correctly selecting $\alpha$ for the efficacy of the method.  A higher $\alpha$, especially early in the optimization process, is desirable in order to increase the speed of convergence.  However later in the process a smaller $\alpha$ is desirable in order to increase the accuracy of the prediction.  This leads to a desire for an $\alpha$ that starts as large as possible and decays at an appropriate rate 
to $0$.

Here we have the additional problem that taking large values of $p$ quickly increases the function values and gradient causing overflow and rounding errors.  However, we see analytically that our accuracy is order $1/p$, so we are conflicted between our desire to take large $\alpha$ and take large $p$.  This makes choosing an $\alpha$ sufficiently (initially) small essential to even begin simulating the problems without errors.    

To solve this issue, we choose (inefficiently but simply) to use very small stepsizes and a large number of iterations.  Due to this choice, we must select $\alpha$ which decay very slowly in order to mitigate the slow convergence caused by taking this small $\alpha$. For instance, taking $\alpha=c/m$ is undesirable in this context since $\alpha$ will be extremely small by the $1000$th iteration, and much more so by the $100,000$th.  Our $\alpha$ of choice are $\alpha(m)=c$ (no decay), $\alpha(m)=c/\sqrt{m}$ (moderate decay) and $\alpha(m)=c/\log(m+1)$ (slow decay).  We then choose $c$ sufficiently small to simulate $p=10$ without immediate errors.

Also to mitigate this issue we rescale the matrices provided in the examples.  By dividing each matrix by its maximum entry we reduce the effect of $p$ in causing overflow errors without altering the solution to the problem.  However, this does not address the root of the issue.

Using an adaptive stepsize algorithm such as a line search could also greatly alleviate this issue, however we do not explore this option here.  In particular, towards the beginning of the process we want a small stepsize to mitigate the extremely large gradients.  Then we want $\alpha$ as large as possible in the middle of the process to quickly approach the region around the minimum where the gradient is no longer explosively large due to the influence of $p$
(but still small enough to avoid errors).  Finally we want $\alpha$ to decay to 0 as it approaches the minimum at the end of the process.

\medskip
\textbf{Example 6.2.}
Here we test the described SGD algorithm \ref{gSVD} for games on a convex one-dimensional problem.  We simulate the gradient descent 1000 times, using 1000 iterations each, to see the convergence of the method.  The results are shown in the form of a histogram in Figure \ref{fig:62}.  We observe nice convergence of the method for the value $p=10$.

\begin{figure}
    \centering
    \textbf{Numerical Results for Example 6.2}\par\medskip
    \includegraphics[scale=.5]{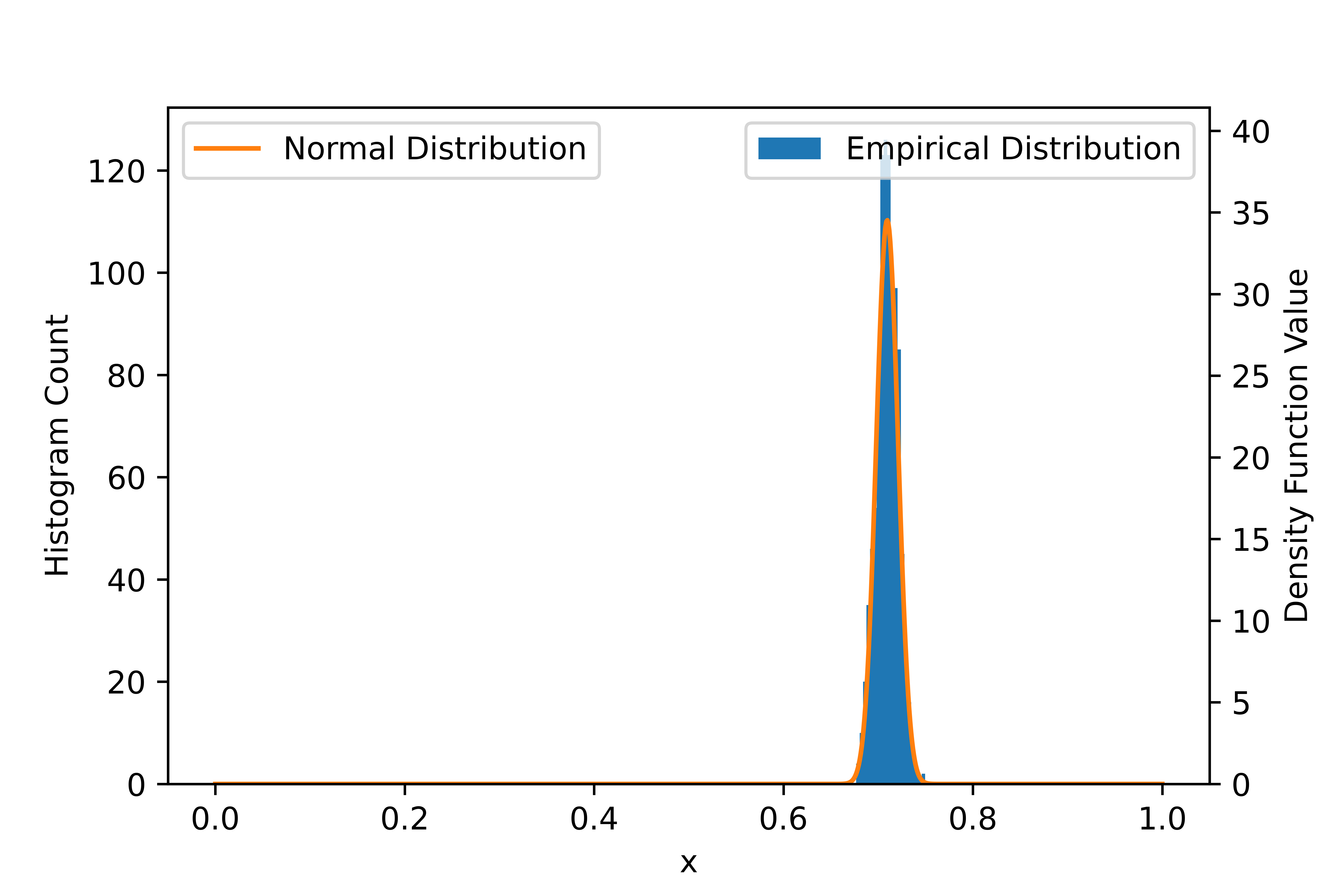}
    \caption{\small
		A histogram plotting 1,000 trials of 1,000 iterations of the described SGD algorithm, along with the density function of a normal distribution of the same expectation and variance.  Here we use $p=10$, $\alpha(m)=.1$.
    }
    \label{fig:62}
\end{figure}

\medskip
\textbf{Example 6.3.}
Here we experiment on a convex two-dimensional problem.  We use this problem to compare the efficacy of the SGD algorithm, in particular the algorithm defined by \eqref{gSVD2}, with full Gradient Descent (GD) and Coordinate Descent (CD).  We also compare to a fourth method which combines SGD with CD.  

Our SGD method relies on expressing the gradient as a summation, then choosing one term of the sum on each iteration to act as an approximation of the gradient.  The CD method uses the full gradient summation, but randomly selects only one coordinate of that gradient while setting the other coordinates to zero.  Thus the combined method chooses one term of the sum to represent the gradient, then sets all directions of the gradient but one to zero.

The results of 1000 trials of each method are shown in the two dimensional histograms in Figures \ref{fig:63_numerics_I} and \ref{fig:63_numerics_II}.  For these experiments we use $p=10$ and stepsize $\alpha(m) = 1/\sqrt{m}$.

We see that when compared at a fixed number of iterations the clear winner is the full gradient descent, followed by CD, then SGD, then the combined method.  This is not surprising since in the context of this problem the SGD method ignores two thirds of the gradient while the CD method ignores half.  The combined method then ignores five sixths of the potential information contained in the gradient.  Ignoring this information serves to greatly increase computational efficiency, however in our tests we use a fixed number of iterations.  This tests the rate of convergence while effectively ignoring the gain in efficiency.  We expect that these efficiency increases are instead more relevant for a high dimensional setting where computation costs are a limiting factor.

\begin{figure}
    \centering
    \textbf{Numerical Results for Example 6.3 Part I}\par\medskip
    {\small Full Gradient Descent} \\
    \begin{subfigure}[b]{0.4\textwidth}
        \centering
        \includegraphics[width=\textwidth]{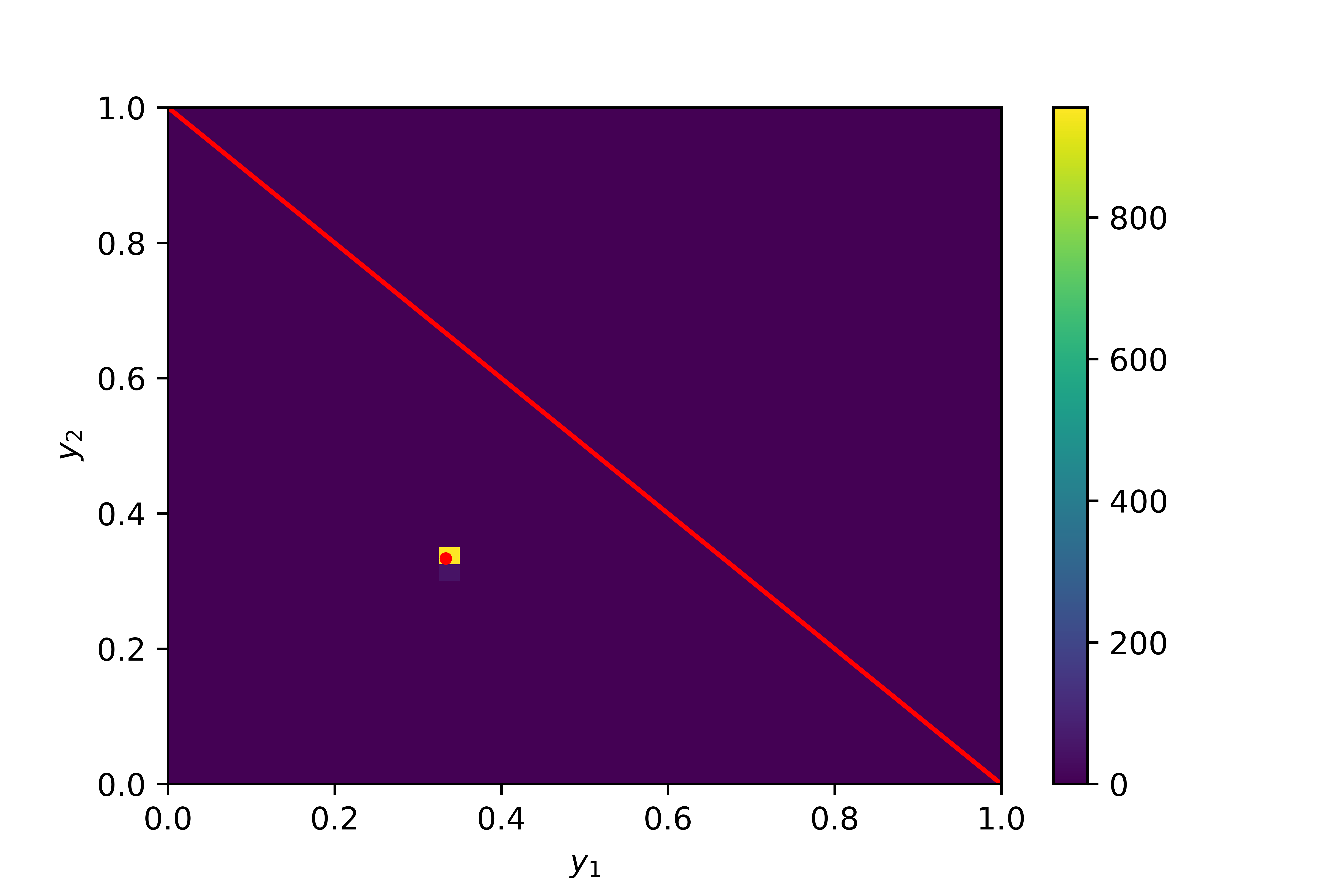}
        \caption[]%
        {{\small 1,000 iterations}}    
        \label{fig:63_GD_1}
    \end{subfigure}
    \hfill
    \begin{subfigure}[b]{0.4\textwidth}  
        \centering 
        \includegraphics[width=\textwidth]{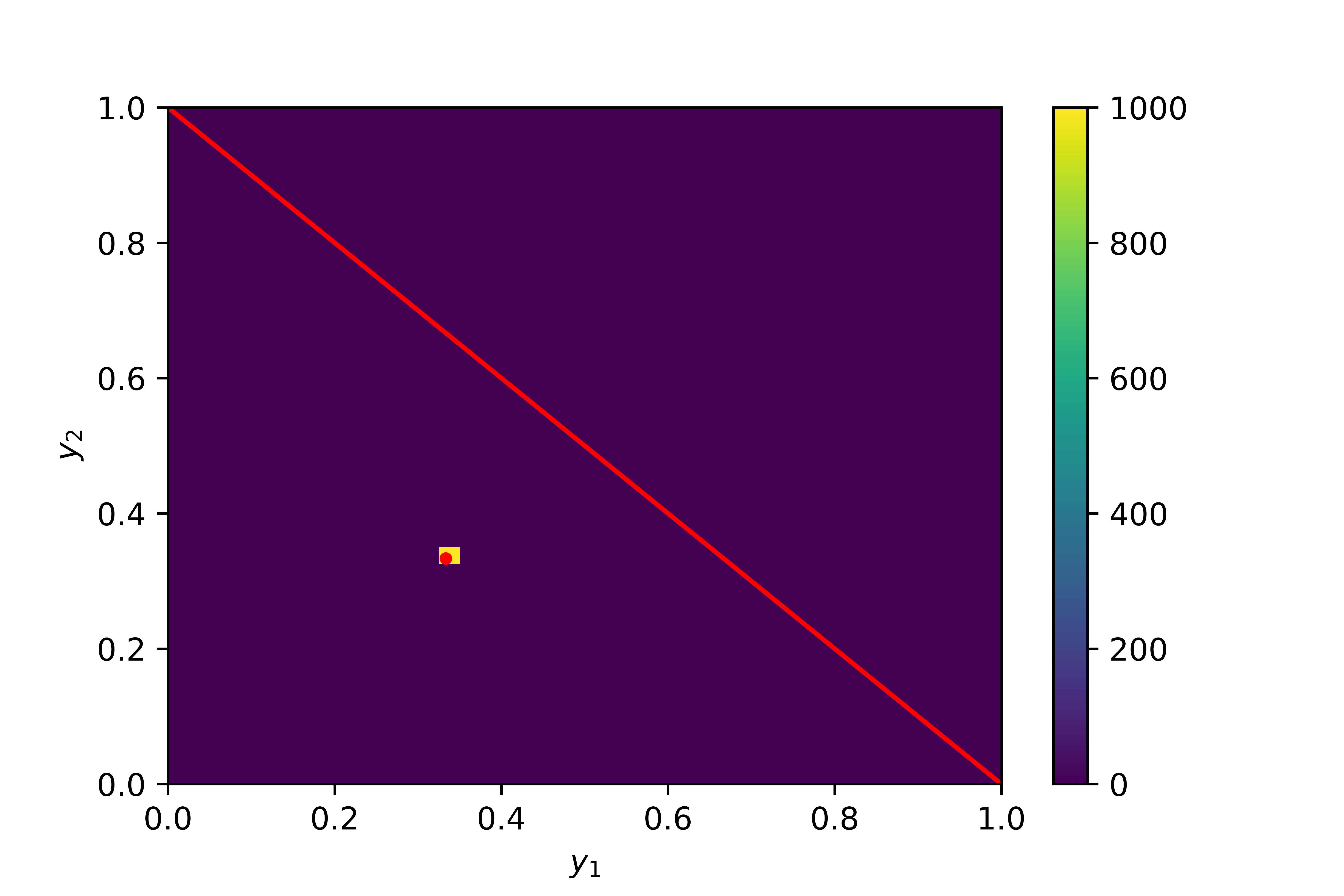}
        \caption[]%
        {{\small 5,000 iterations}}    
        \label{fig:63_GD_2}
    \end{subfigure} \\
    {\small Stochastic Gradient Descent} \\
    \begin{subfigure}[b]{0.4\textwidth}
        \centering
        \includegraphics[width=\textwidth]{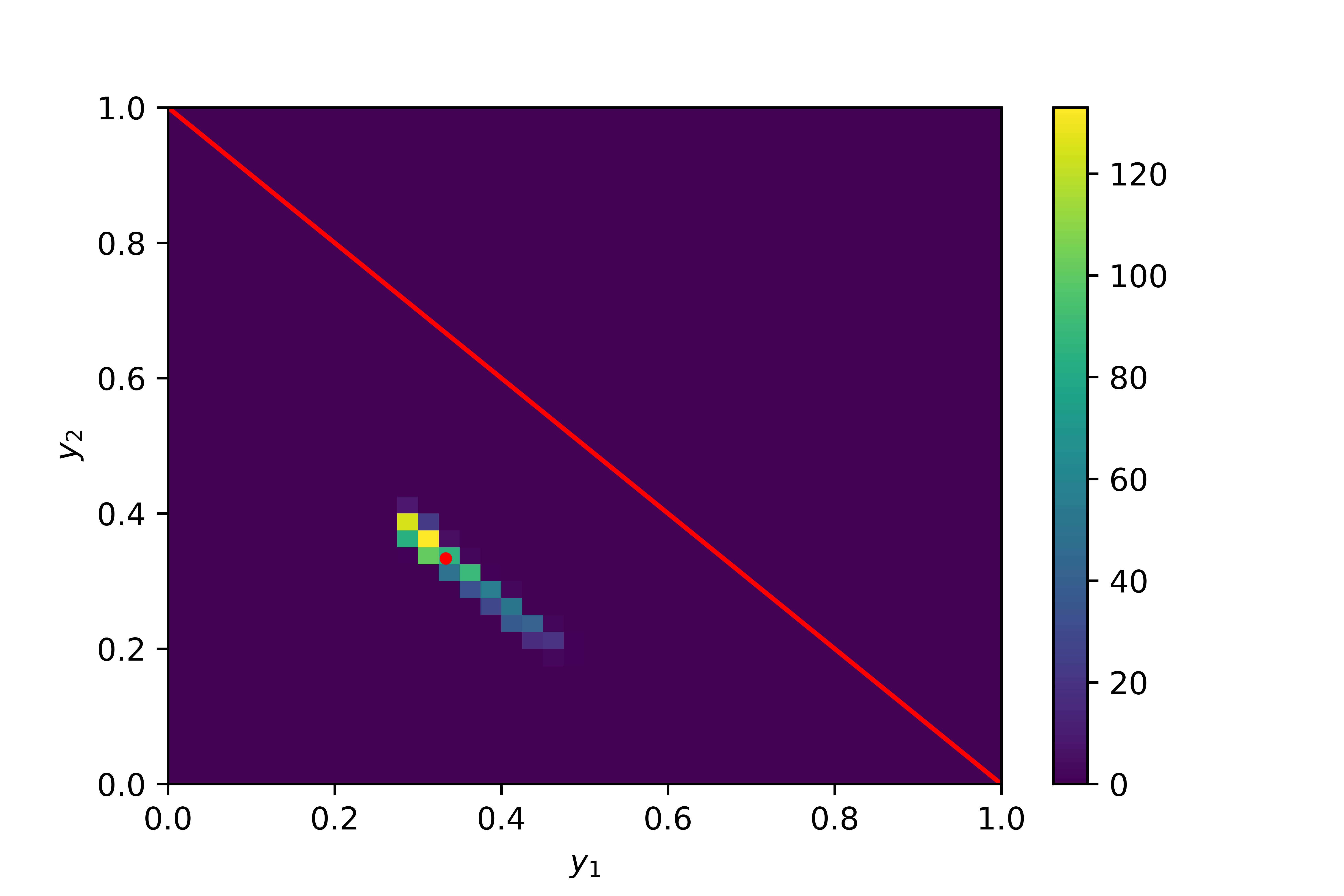}
        \caption[]%
        {{\small 1,000 iterations}}    
        \label{fig:63_SGD_1}
    \end{subfigure}
    \hfill
    \begin{subfigure}[b]{0.4\textwidth}  
        \centering 
        \includegraphics[width=\textwidth]{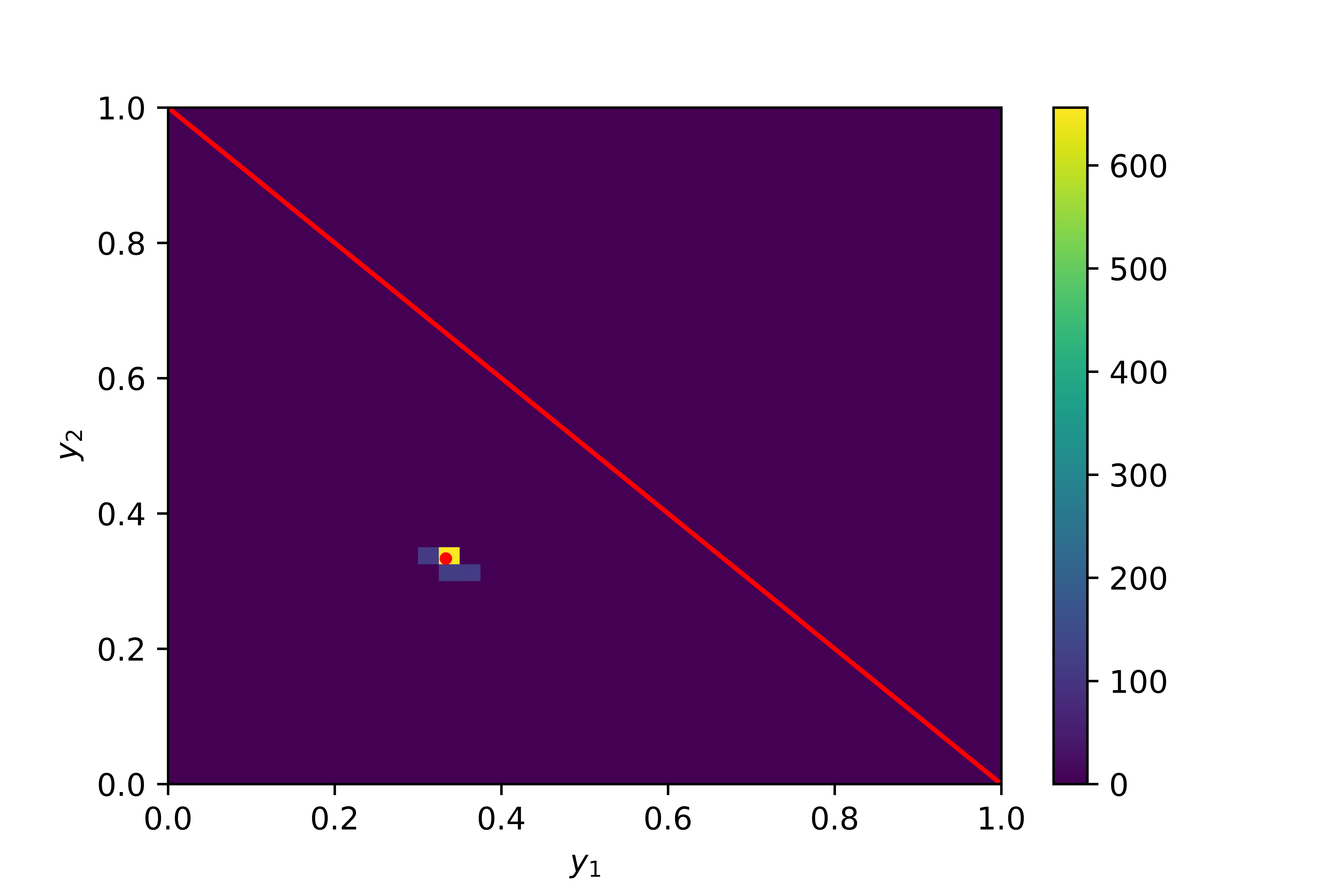}
        \caption[]%
        {{\small 5,000 iterations}}    
        \label{fig:63_SGD_2}
    \end{subfigure}
    \caption[]
    {\small Two dimensional histograms containing the results of 1000 trials of full gradient descent and the described stochastic gradient descent scheme \eqref{gSVD2}.  Here we use $p=10$ and stepsize $\alpha(m) = 1/\sqrt{m}$.  The plotted red point is the analytically found minimum for the exact max function, while the red line is the boundary of the simplex.  Observe the nice convergence to the minimum even at fairly low numbers of iterations.  Compare also with the experiments in figure \ref{fig:63_numerics_II}.}
    \label{fig:63_numerics_I}
\end{figure}

\begin{figure}
    \centering
    \textbf{Numerical Results for Example 6.3 Part II}\par\medskip
    {\small Coordinate Descent} \\
    \begin{subfigure}[b]{0.4\textwidth}
        \centering
        \includegraphics[width=\textwidth]{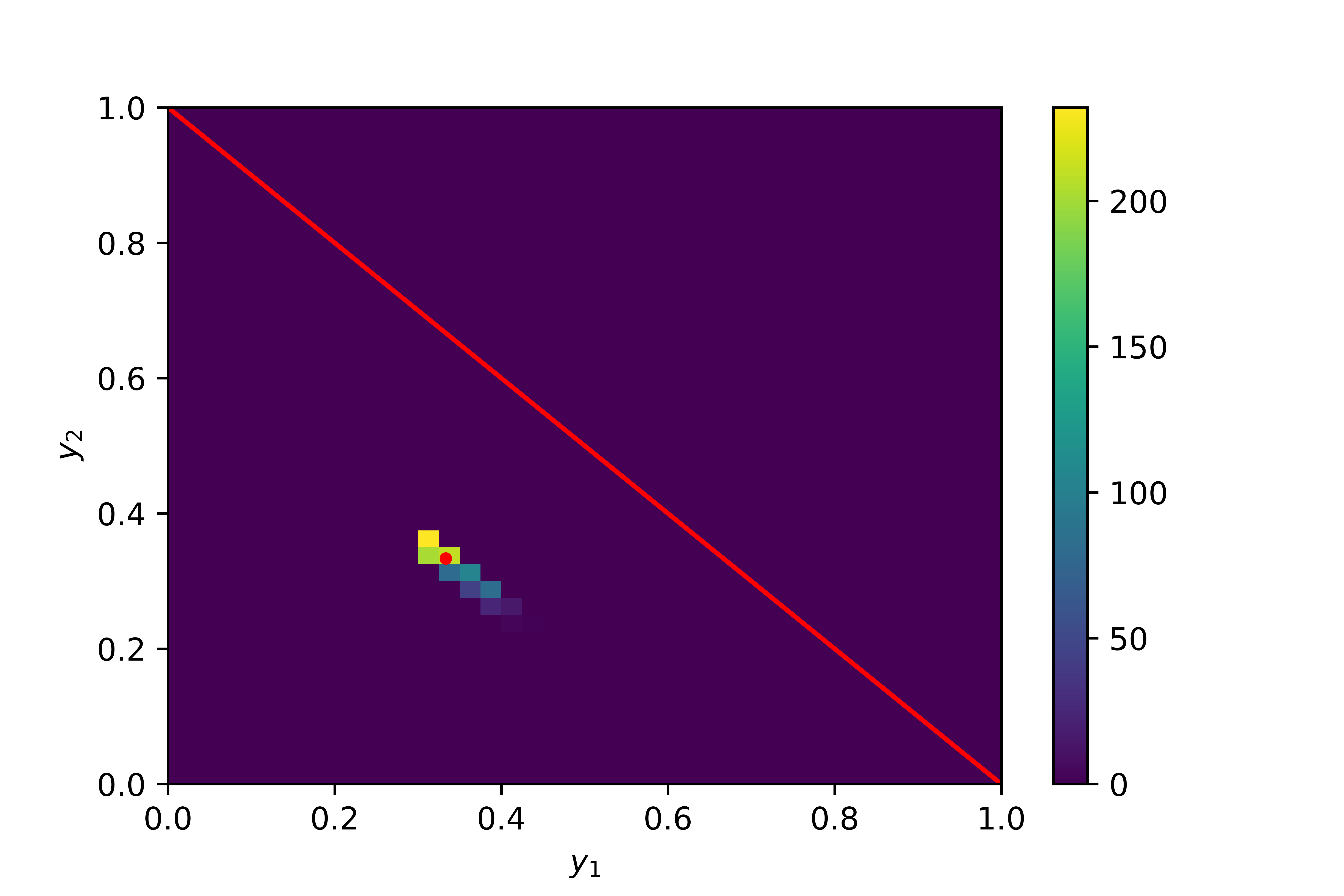}
        \caption[]%
        {{\small 1,000 iterations}}    
        \label{fig:63_CD_1}
    \end{subfigure}
    \hfill
    \begin{subfigure}[b]{0.4\textwidth}  
        \centering 
        \includegraphics[width=\textwidth]{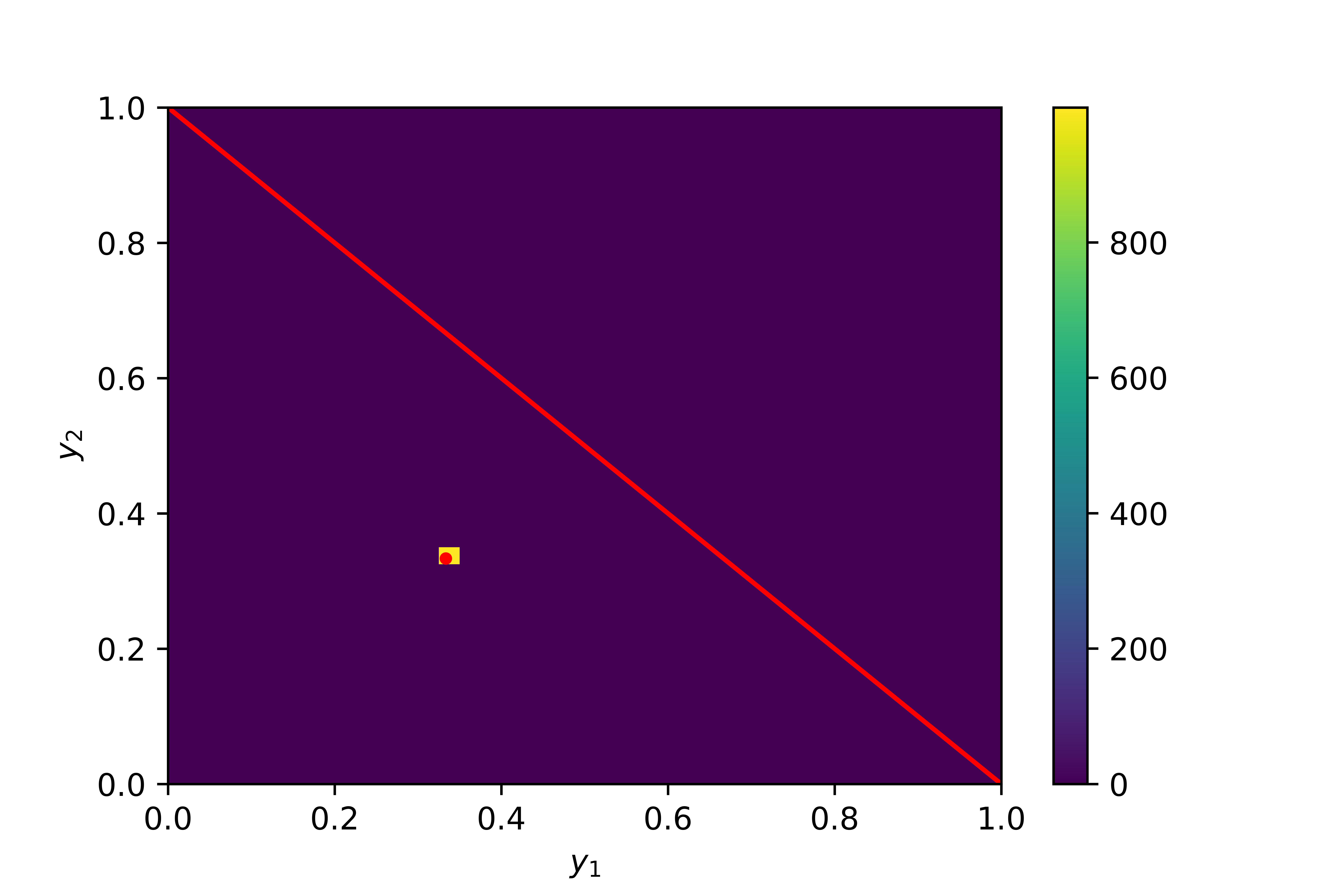}
        \caption[]%
        {{\small 5,000 iterations}}    
        \label{fig:63_CD_2}
    \end{subfigure} \\
    {\small Combined Method} \\
    \begin{subfigure}[b]{0.4\textwidth}
        \centering
        \includegraphics[width=\textwidth]{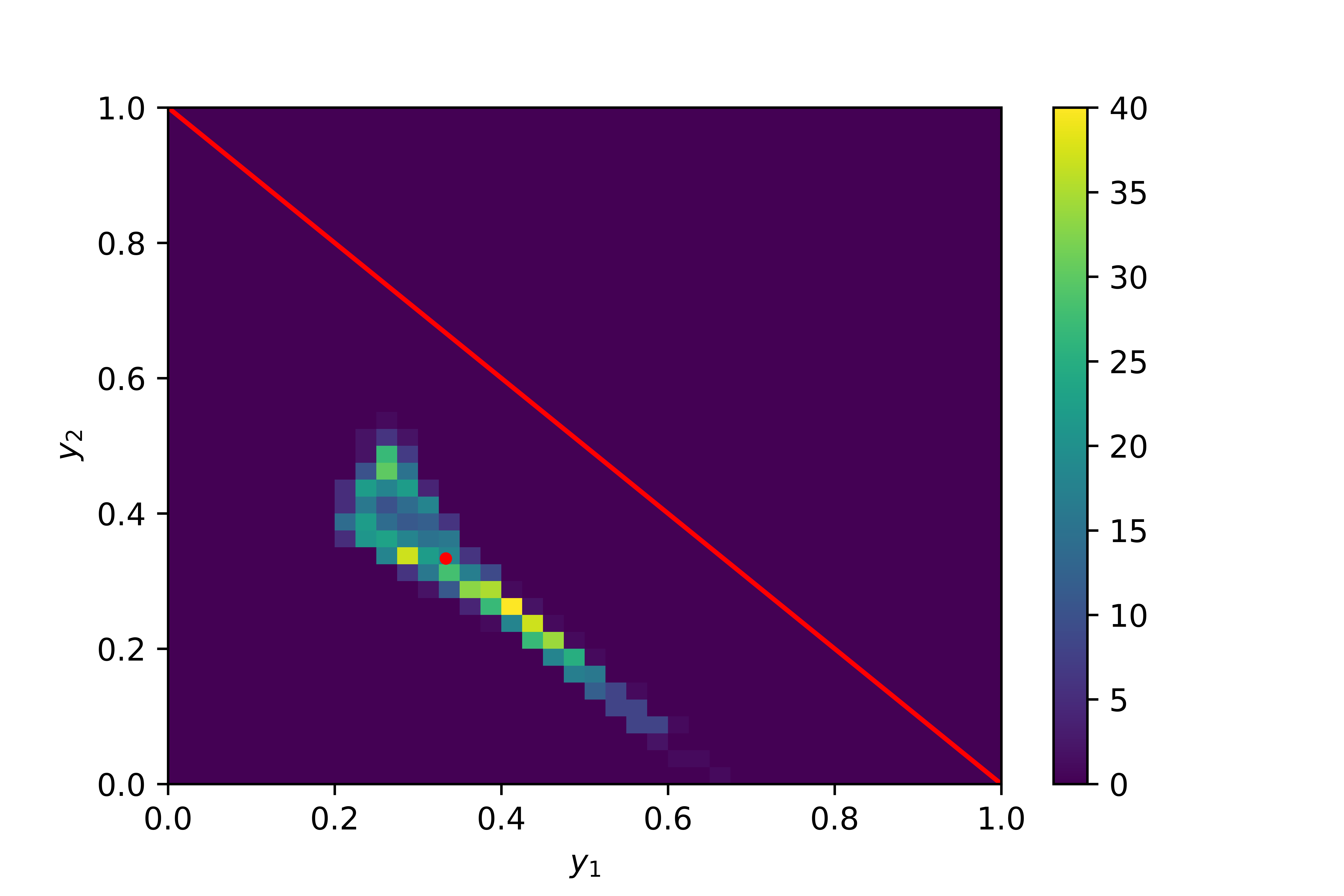}
        \caption[]%
        {{\small 1,000 iterations}}    
        \label{fig:63_both_1}
    \end{subfigure}
    \hfill
    \begin{subfigure}[b]{0.4\textwidth}  
        \centering 
        \includegraphics[width=\textwidth]{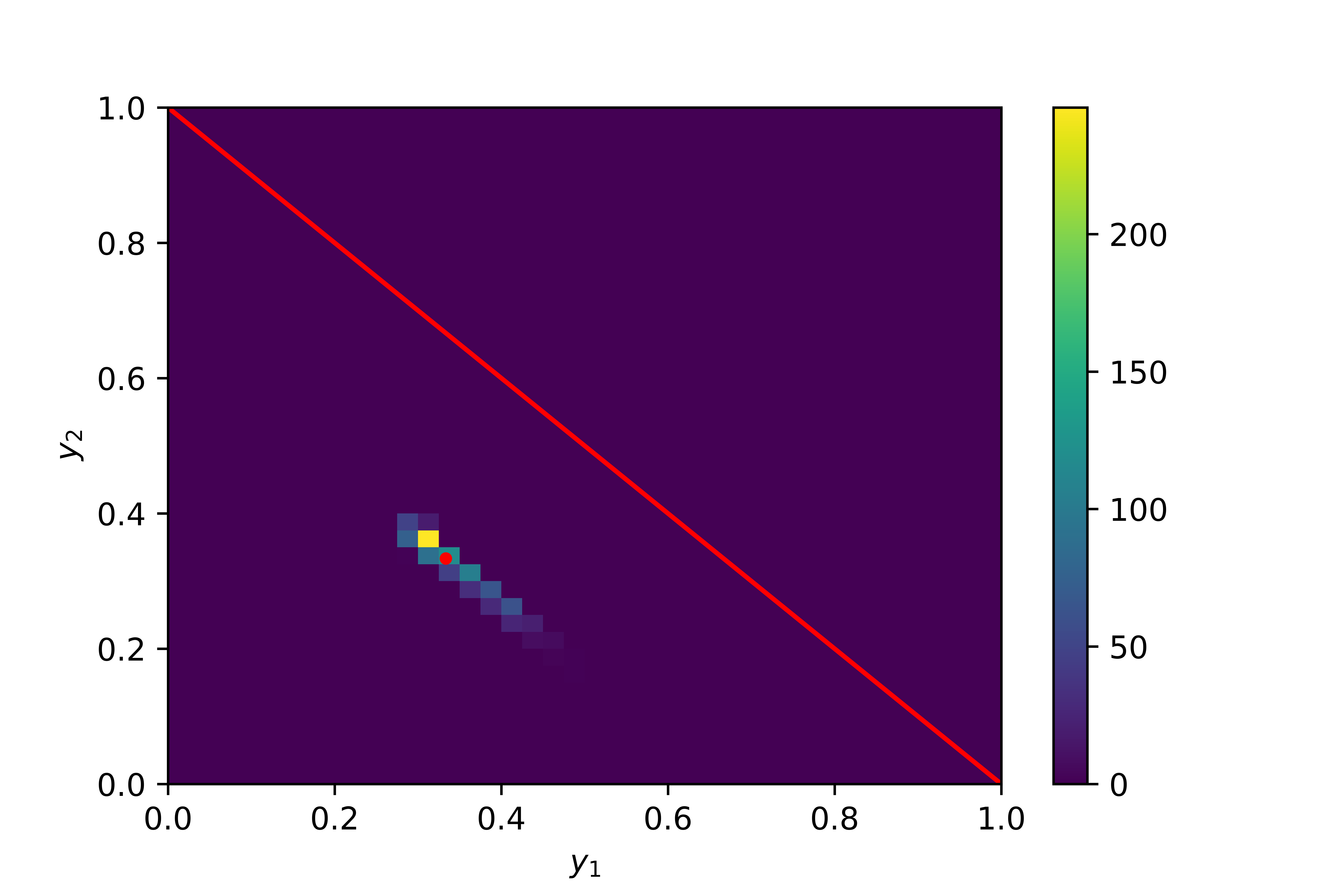}
        \caption[]%
        {{\small 5,000 iterations}}    
        \label{fig:63_both_2}
    \end{subfigure}
    \caption[]
    {\small Two dimensional histograms containing the results of 1000 trials of coordinate descent and a combined coordinate and stochastic gradient descent scheme.  Compare also with the results for Gradient Descent and Stochastic Gradient Descent in figure \ref{fig:63_numerics_I}.  We again use $p=10$ and stepsize $\alpha(m) = 1/\sqrt{m}$.  The plotted red point is the analytically found minimum for the exact max function, while the red line is the boundary of the simplex.  Observe the nice convergence to the minimum as the number of iterations is increased.}
    \label{fig:63_numerics_II}
\end{figure}

\medskip
\textbf{Example 6.4}
Now we test the SGD method on a nonconvex problem in two dimensions.  Additionally, this problem has an interior saddle point and minima outside the domain of the probability space.  We plot the contours for the approximated maximum in figure \ref{fig:64_contour}. 
 Notably the plotted contours are for $p=2$.  As $p$ increases, the gradients become exponentially larger as does the function value. It can be additionally seen in figure \ref{fig:64_contour} that surrounding the local minima are relatively large regions of small gradient.  This potentially can greatly slow the convergence of the method as it approaches the local minima.

\begin{figure}
    \centering
    \textbf{Contour Plot for Approximate Maximum in Example 6.4}
    \includegraphics[scale=.5]{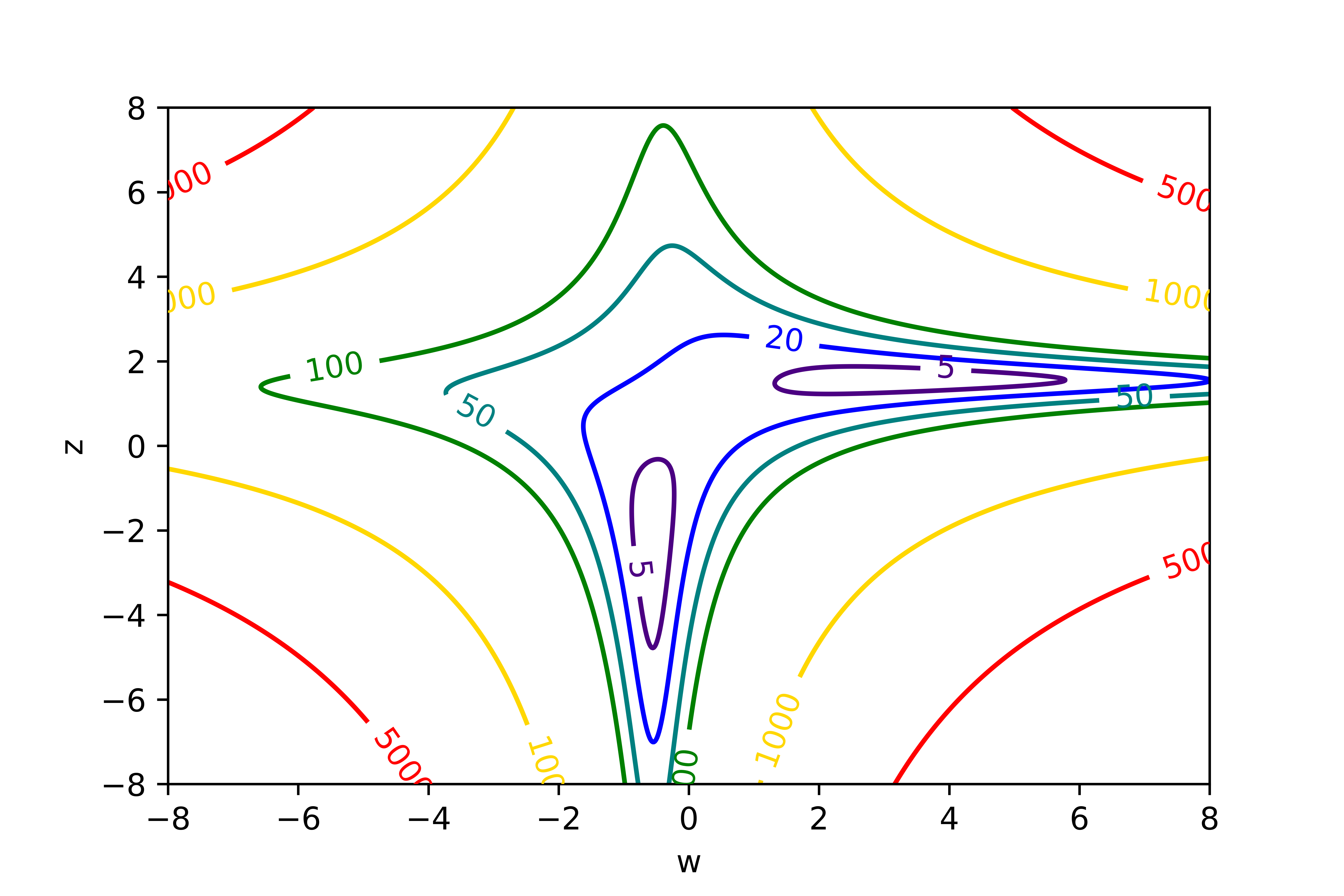}
    \caption{\small
		The contour plot for the approximate maximum function $\phi_1^p + \phi_2^p$ for $\phi_1$, $\phi_2$ given in \eqref{phis}. Here we plot for $p=2$, however the qualitative graph remains the same for larger $p$, but with much higher gradients and function values.
    }
    \label{fig:64_contour}
\end{figure}

Before making any adjustments to the method, we apply the algorithm exactly as described in \eqref{sd2} for $p=10$.  This leads to the results in Figure \ref{fig:64_numerics_large}.  We see the results quickly collapse to a slow manifold (and in particular the saddle point) as can be seen in the contour plot, before slowly moving toward the global minima.  We notice that the convergence slows essentially to a halt before reaching the global minima.  This is due primarily to the large choice of $p$.  This choice of $p$ necessitates very small $\alpha$ because the maximum approximation as well as the corresponding gradient grow exponentially with $p$.  However, there is a large region of relatively low gradient (as can be seen in the figure \ref{fig:64_contour}).  This is combination with the necessarily small choice of $\alpha$ leads to a dramatic slowdown akin to an early termination.  We plot additionally the same simulation for $p=2$ in figure \ref{fig:64_numerics_small} and see that this error does not occur, and convergence to the global minima is achieved.  This highlights the fact that although increasing $p$ increases the fidelity of the approximation to the maximum, it can introduce other errors if the function is irregular in the sense that its gradients vary largely over the domain.

\begin{figure}
    \centering
    \textbf{Numerical Results for Example 6.4 with large $p$}\par\medskip
    \begin{subfigure}[b]{0.475\textwidth}
        \centering
        \includegraphics[width=\textwidth]{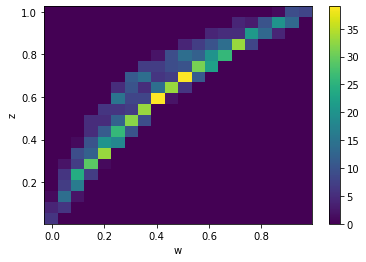}
        \caption[]%
        {{\small 1,000 iterations}}    
        \label{fig:64_large_1}
    \end{subfigure}
    \hfill
    \begin{subfigure}[b]{0.475\textwidth}  
        \centering 
        \includegraphics[width=\textwidth]{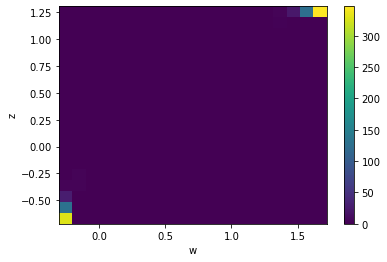}
        \caption[]%
        {{\small 100,000 iterations}}    
        \label{fig:64_large_2}
    \end{subfigure}
    \caption[]
    {\small Two dimensional histograms containing the results of 1000 trials of the described stochastic gradient descent scheme \eqref{sd2}.  Here we use $p=10$ and stepsize $\alpha(m) = 10^{-8}$.  Observe the convergence to two local minima outside the domain.}
    \label{fig:64_numerics_large}
\end{figure}

\begin{figure}
    \centering
    \textbf{Numerical Results for Example 6.4 with small $p$}\par\medskip
    \begin{subfigure}[b]{0.475\textwidth}
        \centering
        \includegraphics[width=\textwidth]{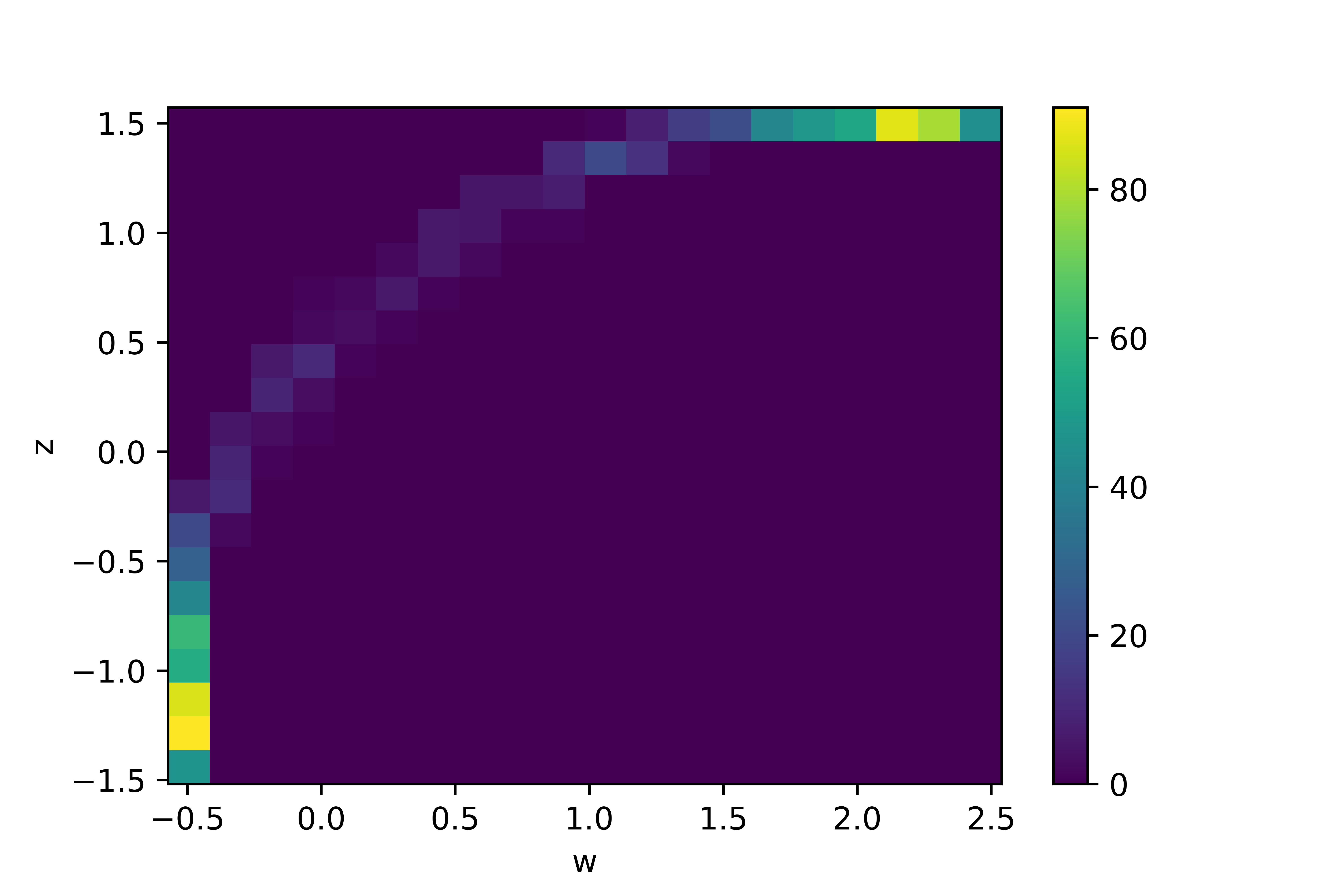}
        \caption[]%
        {{\small 1,000 iterations}}    
        \label{fig:64_3}
    \end{subfigure}
    \hfill
    \begin{subfigure}[b]{0.475\textwidth}  
        \centering 
        \includegraphics[width=\textwidth]{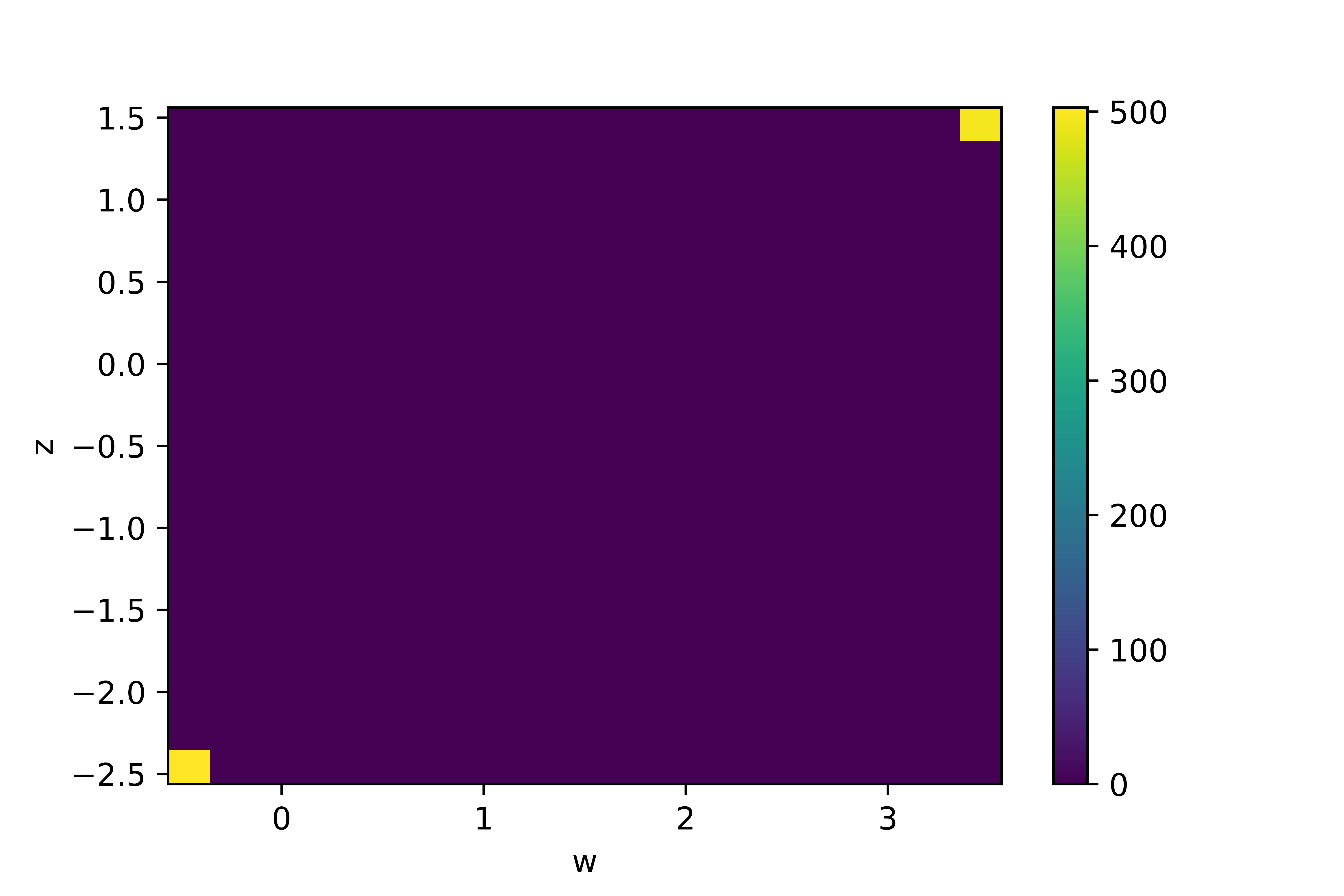}
        \caption[]%
        {{\small 100,000 iterations}}    
        \label{fig:64_4}
    \end{subfigure}
    \caption[]
    {\small Two dimensional histograms containing the results of 1000 trials of the described stochastic gradient descent scheme \eqref{sd2}.  Here we use $p=2$ and stepsize $\alpha(m) = .001$.  Observe the convergence to two local minima outside the domain.}
    \label{fig:64_numerics_small}
\end{figure}

Now we attempt to tackle the issue of the process leaving the desired domain.  Since our problem deals with probabilities we want to avoid the region outside the square $[0,1]\times[0,1]$. To do this we add a simple penalty as described by \eqref{penalty}.  The gradient of this penalty is added to each step of the iteration.  Since it is not multiplied by $\alpha$, it functionally has a very large weight built in, and is effective even with $k=d=1$ (which amounts to relatively weak penalty).  The results of this simulation are shown in Figure \ref{fig:64_penalty}.  We see that with this adjustment the results stay inside the desired domain and find the minima at the corners of the domain.

\begin{figure}
    \centering
    \textbf{Example 6.4 with Penalty}\par\medskip
    \begin{subfigure}[b]{.475\textwidth}
        \centering
        \includegraphics[width=\textwidth]{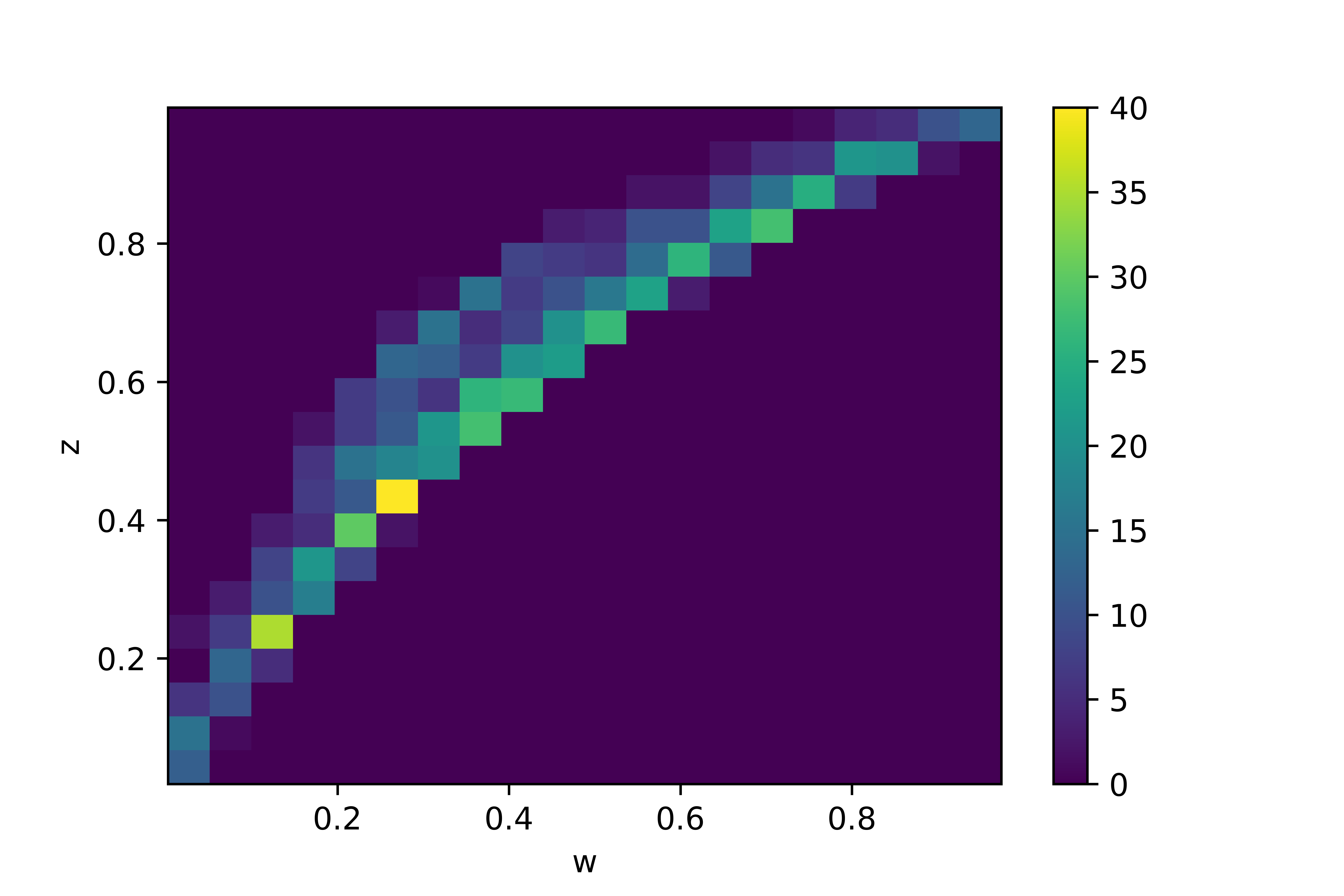}
        \caption[] {{\small 1,000 iterations}}
        \label{fig:64_penalty1}
    \end{subfigure}
    \hfill
    \begin{subfigure}[b]{.475\textwidth}
        \centering
        \includegraphics[width=\textwidth]{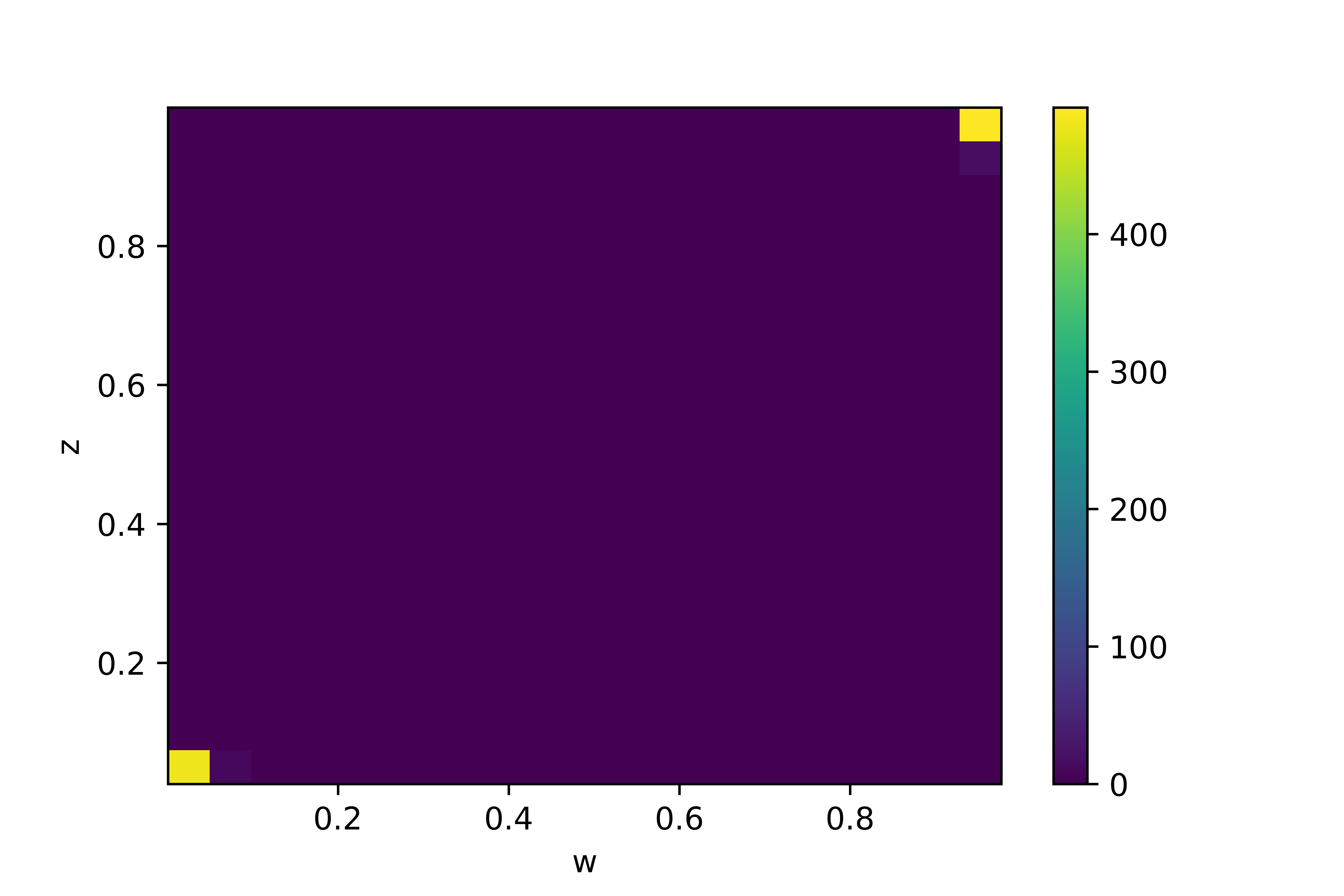}
        \caption[] {{\small 100,000 iterations}}
        \label{fig:64_penalty2}
    \end{subfigure}
    \caption{\small
		Two dimensional histograms containing the results of 1,000 trials of the described stochastic gradient descent scheme \eqref{sd2} with $p=10$, $\alpha = 10^{-8}$ and adjusted with penalty function \eqref{penalty} with $K=d=1$.  We see that with the penalty the results stay in the probability space.
    }
    \label{fig:64_penalty}
\end{figure}

\medskip
\textbf{Example 6.6.}
For the next experiment, we use our described SGD method for a two dimensional convex problem with a unique interior local minimum.  The contour plot of this function can be seen in Figure \ref{fig:66_contour}, where we can see easily the nice properties of this function.  We see that the the results of 1000 trials of SGD can be found in Figure \ref{fig:66_numerics}.  We see the simulation quickly collapse to a normal distribution centered around a point.  
Afterward, we observe that the slow direction (low gradient direction) in the contour plot collapses to the minimum quite quickly, while the fast direction (larger gradient direction) converges more slowly. 
So, interestingly, the band of non-converging points that we observe
is along the line where the gradient is large instead of small.
The shape in Figure \ref{fig:66_2} is present after 1000 iterations, and remains present beyond 100,000.  
We believe that this is due to our choice of constant stepsize in this problem.  We see an oscillation which occurs due to the extremely large gradient causing overshoot.  We leave this simulation as an acknowledgment of the difficulty in choosing $\alpha$ discussed before.

\begin{figure}
    \centering
    \textbf{Contour Plot for Approximate Maximum in Example 6.6}
    \includegraphics[scale=.5]{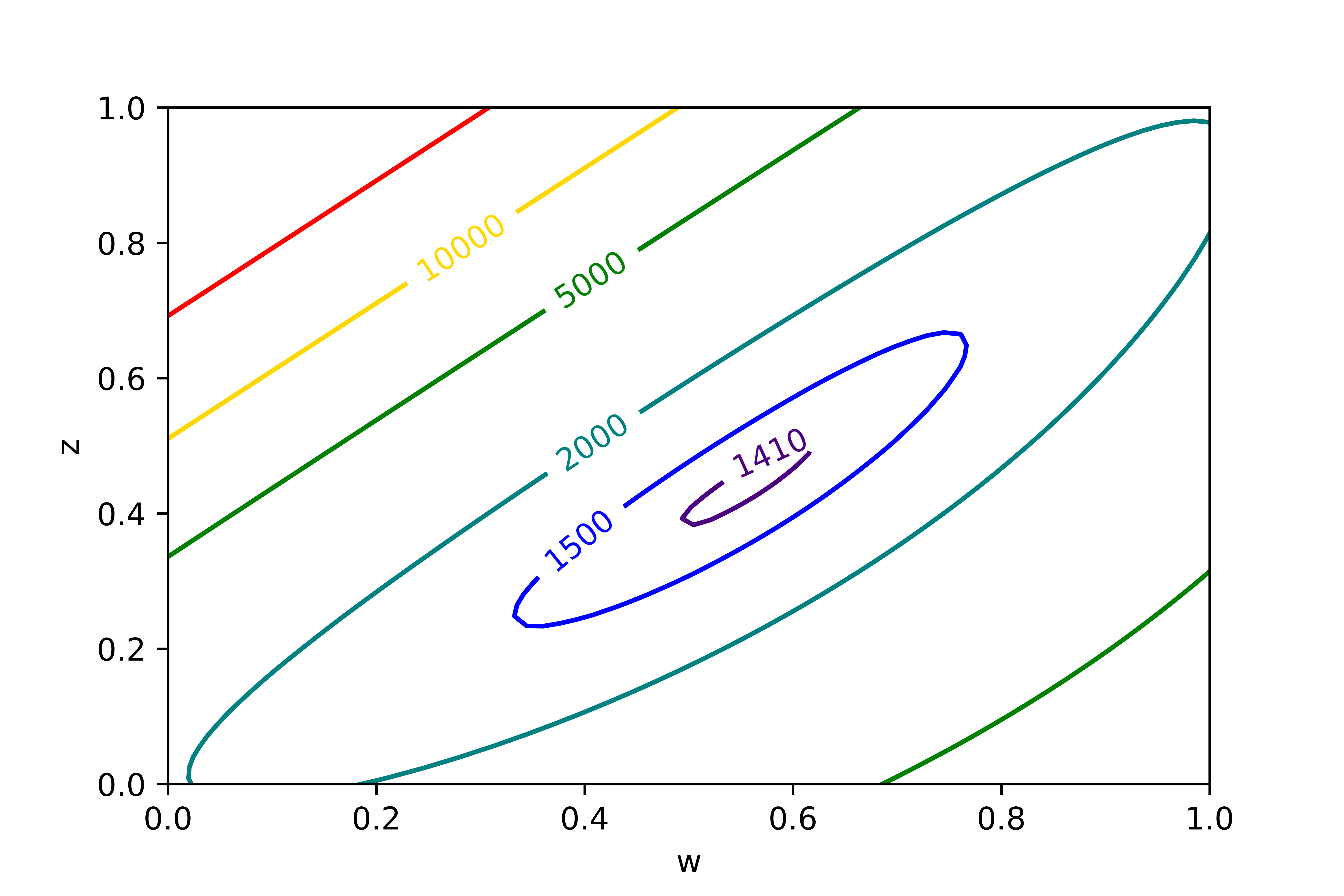}
    \caption{\small
		The contour plot for the approximate max $\phi_1^p + \phi_2^p$ for $\phi_1$, $\phi_2$ given in \eqref{phis2}. Here we plot for $p=10$, which is what we used in the other numerical experiments.
    }
    \label{fig:66_contour}
\end{figure}

\begin{figure}
    \centering
    \textbf{Numerical Results for Example 6.6}\par\medskip
    \begin{subfigure}[b]{0.475\textwidth}
        \centering
        \includegraphics[width=\textwidth]{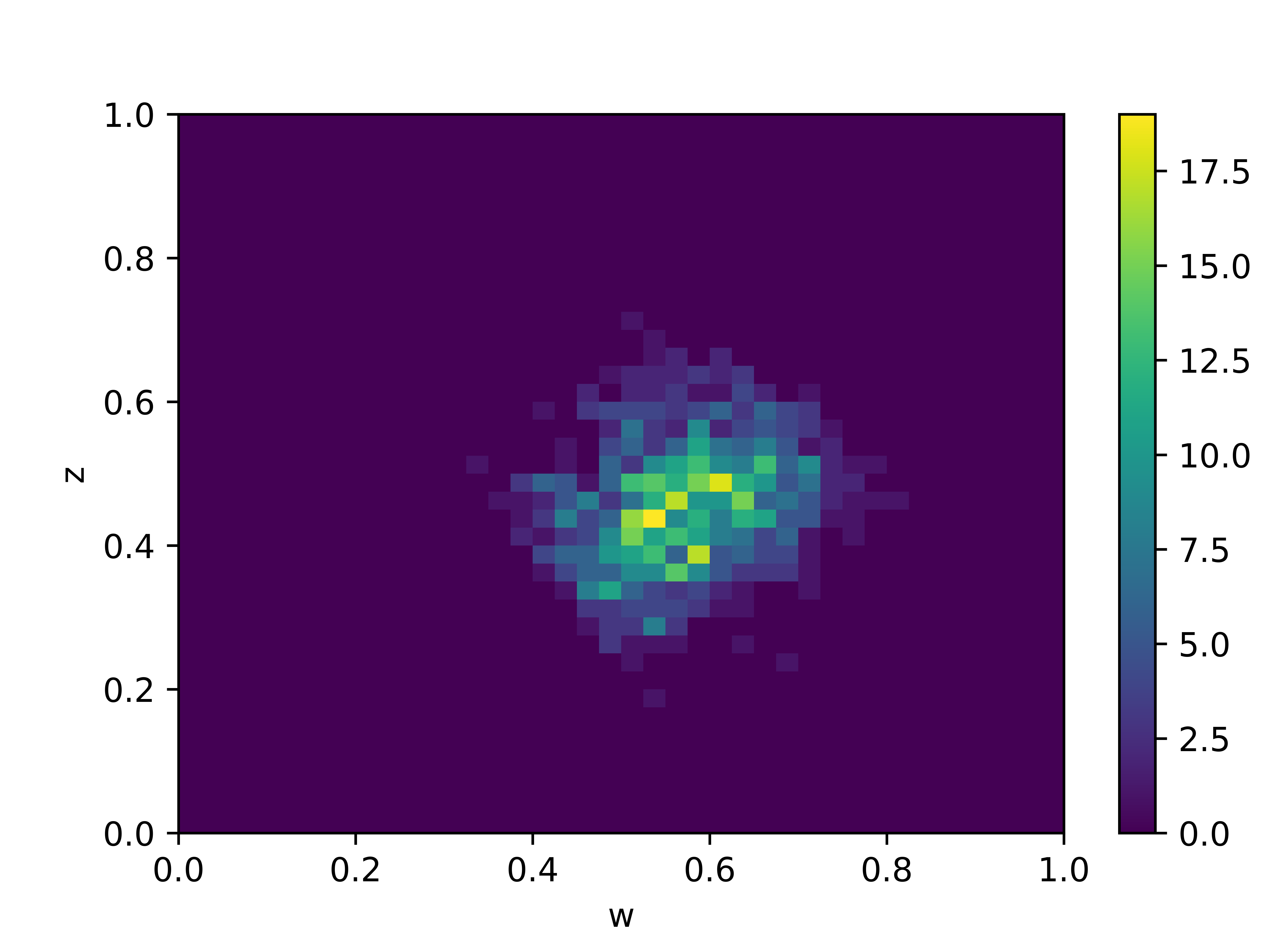}
        \caption[]%
        {{\small 100 iterations}}    
        \label{fig:66_1}
    \end{subfigure}
    \hfill
    \begin{subfigure}[b]{0.475\textwidth}  
        \centering 
        \includegraphics[width=\textwidth]{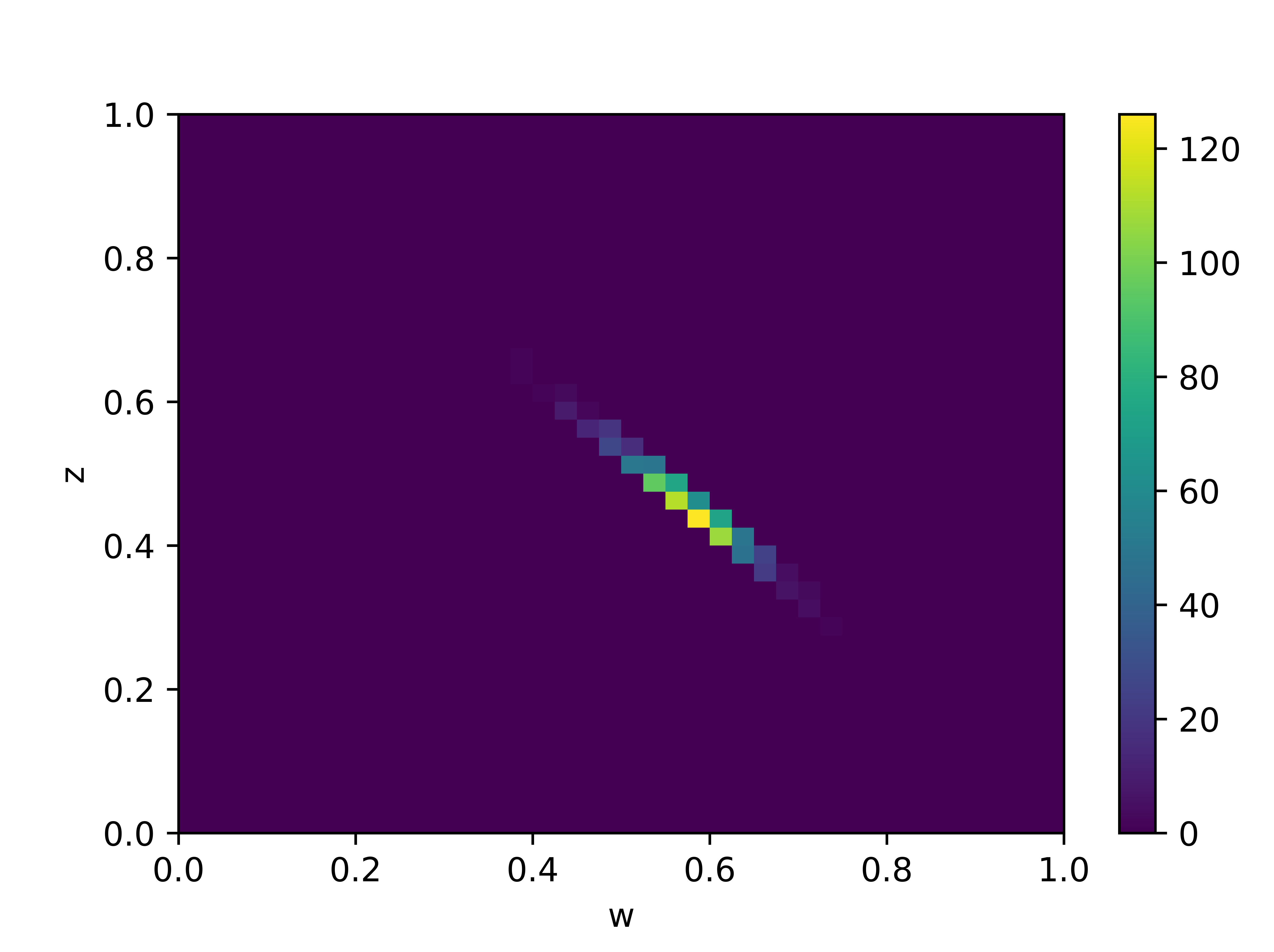}
        \caption[]%
        {{\small 100,000 iterations}}    
        \label{fig:66_2}
    \end{subfigure}
    \hfill
    \caption[]
    {\small Two dimensional histograms containing the results of 1000 trials of the described stochastic gradient descent scheme \eqref{sd3}.  Here we use $p=10$ and stepsize $\alpha(m) = 10^{-5}$. }
    \label{fig:66_numerics}
\end{figure}

By modifying $\alpha(m)$, we can resolve this issue easily, but it requires the knowledge from the previous simulation.  By using $\alpha(m)$ piecewise, where decay starts only at the $1000$th iterations, we can see clear convergence to the local min.  In particular we use
\begin{equation}\label{ex66:modifiedalpha}
    \alpha (m) = \begin{cases}
        c & m\leq1000 \\
        \frac{c}{m-1000} & m>1000.
    \end{cases}
\end{equation}
The results of this simulation are found in Figure \ref{fig:66_modified_alpha}.  Now we see equivalent results at 1000 iterations, while seeing a nice convergence by 5000 iterations: a massive improvement over the previous method.

\begin{figure}
    \centering
    \textbf{Modified $\alpha$ for Example 6.6}\par\medskip
    \begin{subfigure}[b]{0.475\textwidth}
        \centering
        \includegraphics[width=\textwidth]{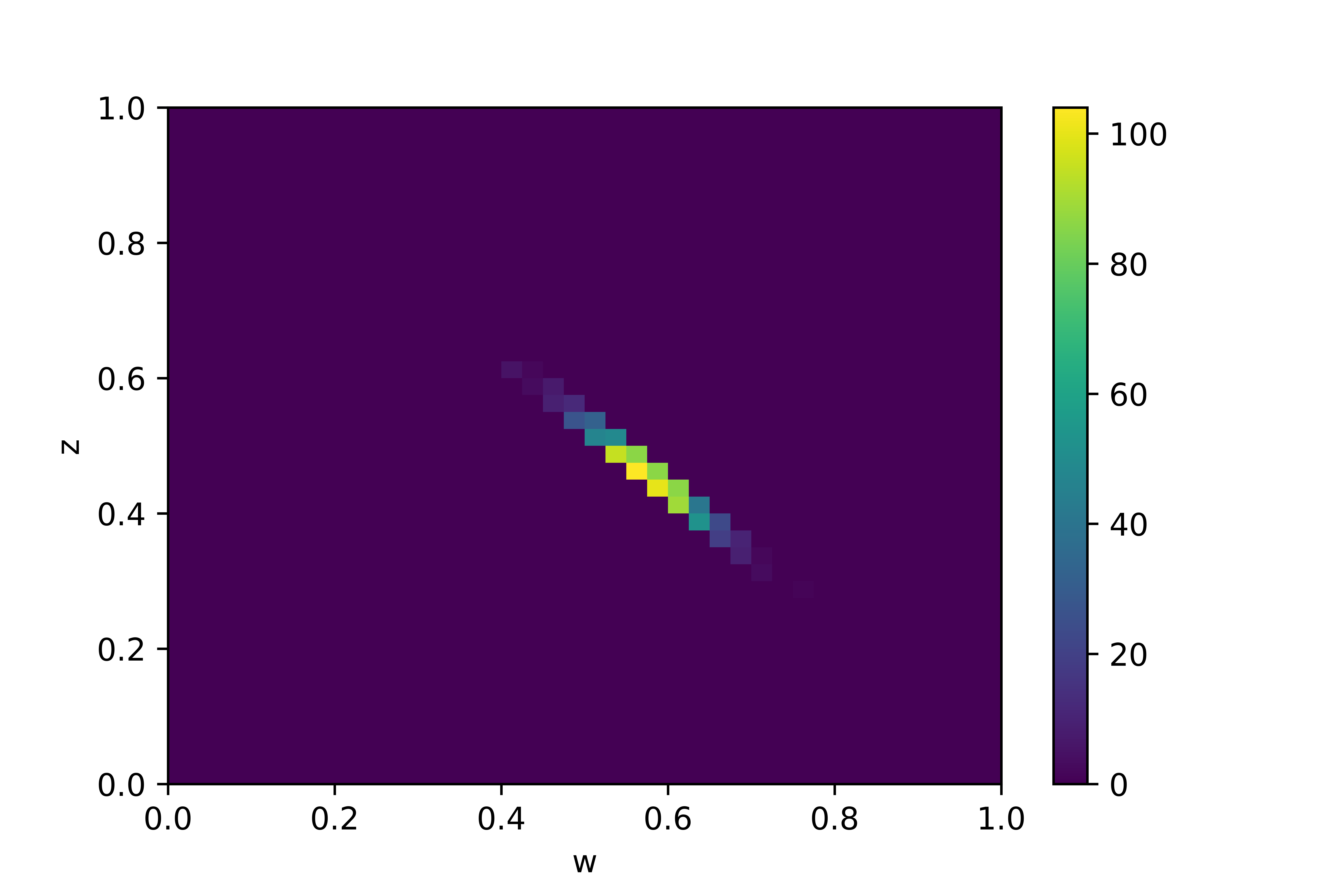}
        \caption[]%
        {{\small 1000 iterations}}    
        \label{fig:66_3}
    \end{subfigure}
    \hfill
    \begin{subfigure}[b]{0.475\textwidth}  
        \centering 
        \includegraphics[width=\textwidth]{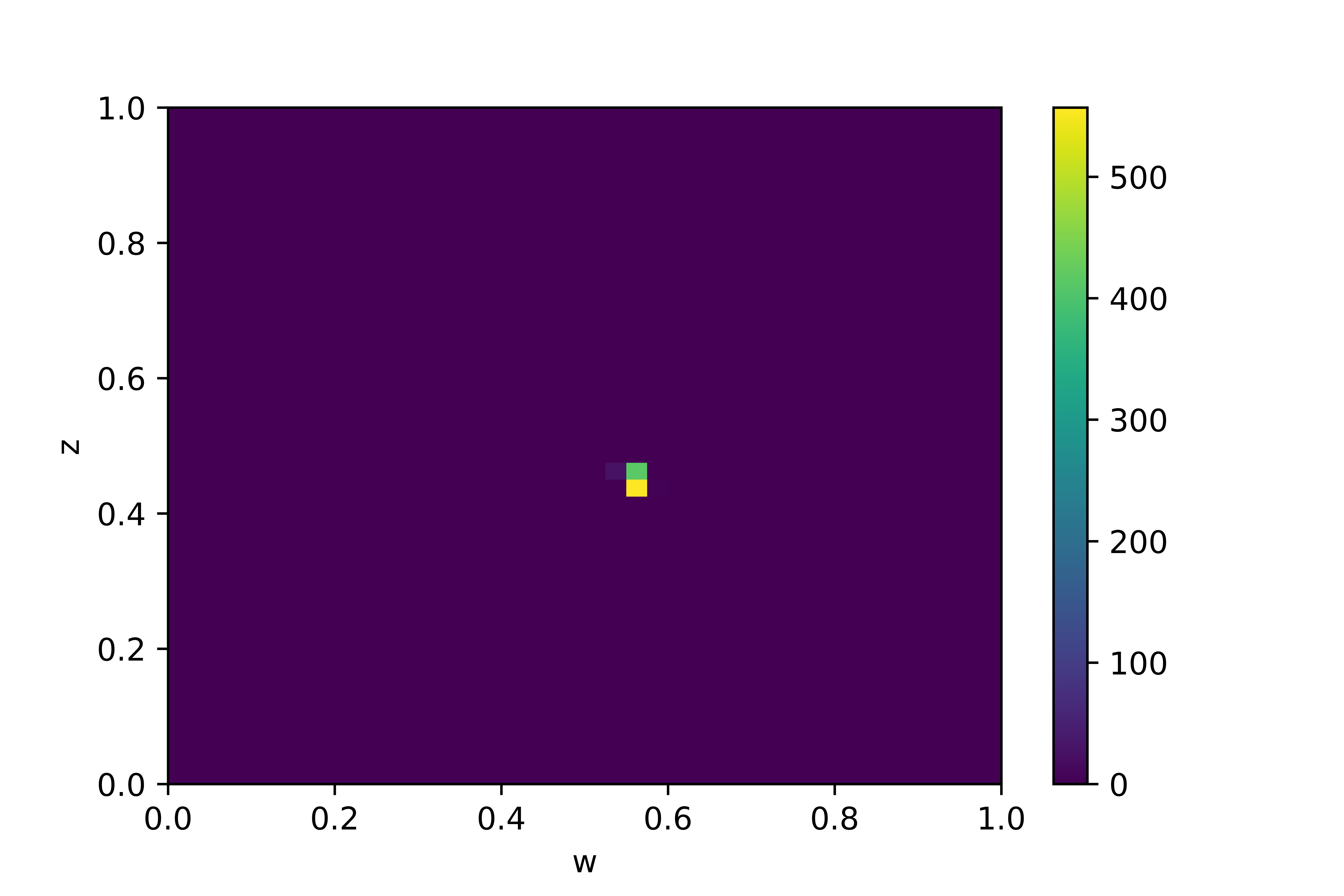}
        \caption[]%
        {{\small 5,000 iterations}}    
        \label{fig:66_4}
    \end{subfigure}
    \hfill
    \caption[]
    {\small Two dimensional histograms containing the results of 1000 trials of the described stochastic gradient descent scheme \eqref{sd3}.  Here we use $p=10$ and stepsize $\alpha(m)$ defined by \eqref{ex66:modifiedalpha}.}
    \label{fig:66_modified_alpha}
\end{figure}

\medskip
\textbf{Example 6.7.}
For this problem, we study the efficacy of SGD on three player rock paper scissors.  This game has several local minima as well as non-unique global minima.  Additionally, there are local and global minima on the boundary of the domain.  Due to the symmetry of the problem and its minima, we restrict player two's strategies by setting $y_1=0$, essentially specifying the orientation with respect to the symmetry.  This leads to the creation of the below Figure \ref{fig:67}.

To generate these figures, we use $\alpha$ defined by
\begin{equation}\label{ex67:alpha}
    \alpha (m) = \begin{cases}
        c & m\leq5000 \\
        \frac{c}{(m-5000)^{.2}} & m>5000.
    \end{cases}
\end{equation}
Additionally, we rescale the matrix by 10 to avoid overflow errors.  This leads to the seemingly strange choice of $c$ for the larger $p$ value, as the gradient is vulnerable to underflow rather than overflow.

\begin{figure}
    \centering
    \textbf{Example 6.7: Rock Paper Scissors}\par\medskip
    \begin{subfigure}[b]{0.475\textwidth}
        \centering
        \includegraphics[width=\textwidth]{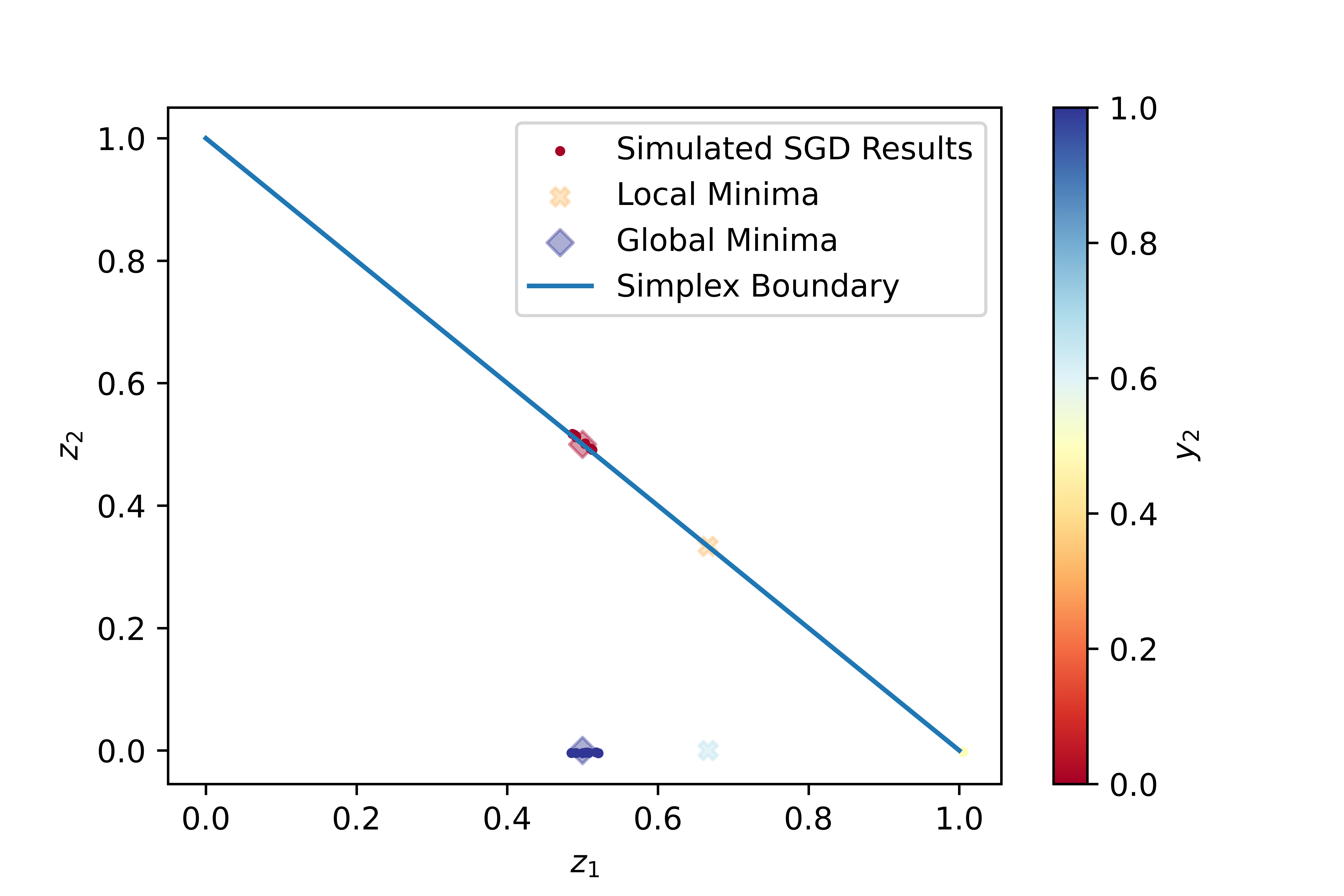}
        \caption[]%
        {{\small $p=2$, $c=.1$}.}    
        \label{fig:67_1}
    \end{subfigure}
    \hfill
    \begin{subfigure}[b]{0.475\textwidth}  
        \centering 
        \includegraphics[width=\textwidth]{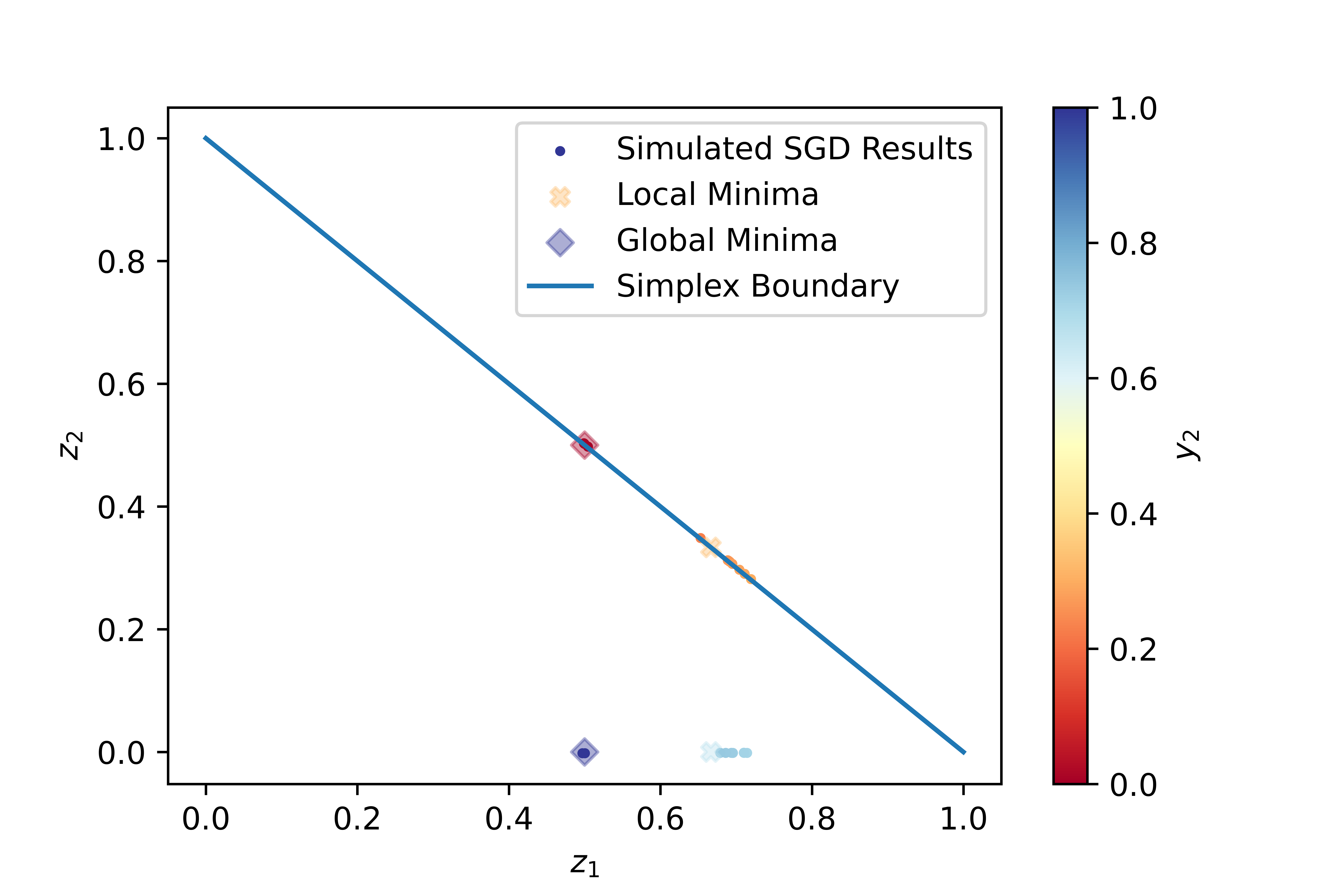}
        \caption[]%
        {{\small $p=10$, $c=10,000$}}    
        \label{fig:67_2}
    \end{subfigure}
    \hfill
    \caption[]
    {\small Here we plot 20 trials of SGD on the Rock Paper Scissors Problem.  For each trial, we perform 100,000 iterations using $\alpha$ defined by \eqref{ex67:alpha}.  Additionally, we plot the local and global minima for the exact problem.  Notice that for both values of $p$, the found minima for the smoothed problem also correspond exactly to minima for the exact non-smoothed problem.  For $p=2$, 4 of the 20 trials converge to the local min (2/3,2/3,0), 8 trials converge to the global min (1,1/2,0), and 8 trials converge to the global min (0,1/2,1/2).  For $p=10$, 8 of the 20 trials converged to the local min (1/3, 2/3, 1/3), 7 of the trials converge to the local min (2/3, 2/3, 0), 2 of the trials converge to the global min (1, 1/2, 0), and 3 of the trials converge to the global min (0, 1/2, 1/2).
    }
    \label{fig:67}
\end{figure}

We see nice convergence to the minima of the exact problem for both the $p=2$ and $p=10$ cases.  Notice especially that for both smoothing values the solutions to the exact non-smoothed problem are found.  For the $p=2$ case we see that only four of the trials converge to a local minima, indicating that they are largely smoothed out of the problem.  For $p=10$ we see many more trials convergence to local minima, though some still converge to the global minima.

\medskip
\textbf{Concluding summary.} 
We see that each of stochastic gradient descent, stochastic coordinate descent, and combined stochastic
gradient-coordinate descent is feasible for our formulation \eqref{cang} of 2- and many-player asynchronous games.
However, the choice of proper scaling and step size is somewhat delicate; likewise, error decreases somewhat
slowly with respect to the smoothing parameter $p$. Both of these suggest that some kind of multigrid iteration
would be essential for scaling up to large-$N$ problems, with $p$, rescaling of the objective, and step size
dynamically modified with successive iterations. A possible side-benefit, illustrated by the results for example
6.7, is that oversmoothing at initial steps ($p$ small) can remove local minima, acting as a sort of annealing
in the process.  Such a treatment of larger games, and systematic comparisons of computational cost,
would be very interesting directions for future exploration.

\end{document}